\documentclass[11pt]{amsart}
\usepackage{amssymb,amsfonts,graphicx,amsmath,amsthm,amstext,latexsym,amsfonts}
\usepackage{graphicx,epsfig}
\usepackage{bbm}
\usepackage{t1enc}
\usepackage{mathrsfs}
\usepackage{subfig}
\usepackage{hyperref}
\usepackage{a4wide}
\usepackage{tikz}
\usetikzlibrary{intersections}
\usepackage{egothic}
\usepackage [T1] {fontenc}
\theoremstyle{plain}

\numberwithin{equation}{section}
\newtheorem{theorem}{Theorem}[section]
\newtheorem{lemma}[theorem]{Lemma}
\newtheorem{remark}[theorem]{Remark}

\usepackage{color}
\definecolor{darkred}{rgb}{0.8,0,0}
\definecolor{darkblue}{rgb}{0,0,0.7}
\definecolor{darkgreen}{rgb}{0,0.4,0}

\newcommand{\eps}{\varepsilon}

\newcommand{\R}{{\mathbb R}}

\newcommand{\W}{{\mathcal W}}

\newcommand{\V}{{\mathcal V}}

\newcommand{\un}{{\rm 1\kern -2.5pt l}}
\newcommand{\tr}{{\rm Tr}}

\def\w{\mathbf{w}}
\def\u{\mathbf{u}}

\def\vv{\mathbf{v}}
\def\yy{\mathbf{y}}
\def\n{\mathbf{n}}

\def\xx{\mathbf{x}}
\def\aa{\mathbf{a}}

\def\vv{\mathbf{v}}

\def\eps{\varepsilon}
\def\R{{\mathbb R}}

\def\H{{\mathcal H}}

\def\eps{\varepsilon}
\def\R{{\mathbb R}}

\def\H{{\mathcal H}}

\def\F{{\mathcal F}}

\def\E{{\mathbb{E}}}

\def\argmin{\mathop{{\rm argmin}}\nolimits}

\def\Tr{\mathop{{\rm Tr}}\nolimits}

\def\dv{\mathop{{\rm div}}\nolimits}

\def\u{\mathbf{u}}
\def\v{\mathbf{v}}

\def\z{\mathbf{z}}
\def\v{{\bf v}}
\def\w{{\bf w}}

\def\x{{\bf x}}

\def\Rot{\mathbf{R}}
\def\Id{\mathbf{I}}

\def\wconv{\rightharpoonup}

\newcommand{\KKK}{\color{black}}

\renewcommand{\epsilon}{\varepsilon}
\newcommand{\beeq}{\begin{equation}}
\newcommand{\eneq}{\end{equation}}
\newcommand{\bear}{\begin{array}}
\newcommand{\enar}{\end{array}}
\newcommand{\bema}{\begin{displaymath}}
\newcommand{\enma}{\end{displaymath}}

\newcommand{\beea}{\begin{eqnarray}}
\newcommand{\enea}{\end{eqnarray}}

\newcommand{\om}{\Omega}

\newcommand{\bb}{\boldsymbol}

\newcommand{\lab}[1]{ \label{#1} }

\newenvironment{proofth1}{\removelastskip\par\medskip   
\noindent{\bf Proof of {\rm {\bf Theorem \ref{mainth1}}}.}
\rm}{\penalty-20\null\hfill$\blacksquare$\par\medbreak} 

\def\Rot{\mathbf{R}}
\def\Id{\mathbf{I}}

\def\wconv{\rightharpoonup}

\title[Linearization of pure traction problems in incompressible elasticity]{Variational linearization of pure traction problems in incompressible elasticity}
  \author{Edoardo Mainini}
\address[Edoardo Mainini]{Dipartimento di Ingegneria meccanica, energetica, gestionale e dei trasporti, 
  Universit\`a  degli studi di Genova, Via all'Opera Pia, 15 - 16145 Genova Italy.}
\email{mainini@dime.unige.it}
\urladdr{http://www.dime.unige.it/it/users/edoardo-mainini}

\author{Danilo Percivale}
\address[Danilo Percivale]{Dipartimento di Ingegneria meccanica, energetica, gestionale e dei trasporti, 
  Universit\`a  degli studi di Genova, Via all'Opera Pia, 15 - 16145 Genova Italy.}
\email{percivale@dime.unige.it}

\begin{document}
 \maketitle
\begin{abstract}
 We consider pure traction problems and we show that incompressible linearized elasticity can be obtained as variational  limit of incompressible finite elasticity   under suitable conditions on external loads. \KKK
\end{abstract}


\section{Introduction}

 Let us consider a hyperelastic body occupying a bounded open set $\Omega\subset\mathbb R^3$ in its reference configuration. 
Equilibrium states under  a body force field ${\mathfrak f}: \Omega\to\mathbb R^3$ and a surface force field ${\mathfrak g}:\partial\Omega\to\mathbb R^3$ are obtained by minimizing  the total energy   \[
 \int_\Omega\mathcal W^I(x,\nabla\mathbf y(x))\,dx-\int_\Omega (\mathbf y(x)-x)\cdot{\mathfrak f}(x)\,dx-\int_{\partial\Omega}(\mathbf y(x)-x)\cdot{\mathfrak g}(x)\,d\mathcal H^2(x).
 \]
 Here, $\mathbf y:\Omega\to\mathbb R^3$ denotes the deformation field,   $\mathcal H^2$ denotes the surface measure, and
$\mathcal W^I:\Omega\times\mathbb R^{3\times 3}\to[0,+\infty]$  is   the incompressible strain energy density. We require incompressibility by letting $\mathcal W^I(x,\mathbf F)=+\infty$ whenever $\det \mathbf F\not =1$.
Moreover, we assume that 
  $\mathcal W^I(x,\cdot)$ 
is a frame indifferent function  that is minimized at $\mathbf F=\mathbf I$ with value $0$. 
 
%
 \bigskip


If  $h > 0$ is an adimensional small parameter, we rescale the displacement field and the external forces by letting $\mathfrak f= h\mathbf f$, $\mathfrak g=h\mathbf g$ and   $\mathbf y(x)-x=h \mathbf \v(x)$. We get
\[
\mathcal E^I_h(\v):=\displaystyle \int_\om \W^I(x, \mathbf I+h\nabla \v)\,dx-h^2\int_\om\mathbf f\cdot\mathbf v\,dx-h^2\int_{\partial\Omega}\mathbf g\cdot\v\,d\mathcal H^2(x).\]
  We aim at obtaining the behavior of rescaled energies  as $h\to 0$ and at showing that the linearized elasticity functional arises in the limit. More precisely, we aim at  proving that   \begin{equation}\label{heuristic} \inf\mathcal E^I_h= h^2\min\mathcal E^I +o(h^2),\end{equation}
  and that if
   $\mathcal E^I_h(\v_h)-\inf \mathcal E^I_h=o(h^2)$ (i.e., if $\v_h$ is a sequence of almost minimizers for $\mathcal E^I_h$) then \begin{equation}\label{heuristic2}\v_h\to \v_0\in \argmin\mathcal E^I\end{equation} in a suitable sense, where
 \begin{equation*}
\mathcal E^I(\v):=
 \displaystyle\int_\Omega\mathcal Q^I(x,\mathbb E(\v))\,dx -\int_\om\mathbf f\cdot\mathbf v\,dx-\int_{\partial\Omega}\mathbf g\cdot\mathbf v\,d\mathcal H^2(x).
 \end{equation*}
 Here, $\E(\vv):= \tfrac12(\nabla\v+\nabla \v^T)$ is the infinitesimal strain tensor field and
 $\mathcal Q^I(x,\cdot)$ is  defined for every  $\mathbf F\in\R^{3\times3}$  by
 $$\mathcal Q^I(x,\mathbf F):=\lim_{h\to 0}h^{-2}\mathcal W^I(x, \exp(h\mathbf F))=\left\{\begin{array}{ll}\frac12\,\mathbf F^T\,D^2 {\mathcal W}(x,\mathbf I)\,\mathbf F\quad&\mbox{if $\mathrm{Tr}\,\mathbf F=0$},\\
 +\infty\quad&\mbox{if $\mathrm{Tr}\,\mathbf F\neq 0,$}
 \end{array}\right.
 $$
 where ${\mathcal W}(x,\mathbf F):=\mathcal W^I(x, (\det \mathbf F)^{-1/3}\,\mathbf F)$  is the isochoric part of $\mathcal W^I$. \KKK
 Such a quadratic form is 
 obtained by a formal Taylor expansion around the identity matrix  ($D^2$  denoting  the Hessian in the second variable). 
 Since $\mathcal Q^I(x,\mathbf F)=+\infty$ if $\mathrm{Tr}\,\mathbf F\neq 0$, we see that $\mathcal E^I(\v)$ is finite only if $\dv\v=0$ a.e. in $\Omega$. Therefore, $\mathcal E^I$ is the linearized elastic energy with elasticity tensor $D^2{\mathcal W}(x,\mathbf I)$ and $\dv\v=0$ is
  the linearized 
 incompressibility constraint.
 
 	\bigskip

 
 \medskip
 
 Under Dirichlet boundary conditions, \eqref{heuristic}-\eqref{heuristic2} have been obtained in \cite{MP}, by means of a  $\Gamma$-convergence analysis with respect to the weak topology of $W^{1,p}(\Omega,\mathbb R^3)$, where the exponent $p$ is suitably related to the coercivity properties of $\mathcal W^I$ (see Section \ref{sectmain}).
 On the other hand, in this paper we shall consider natural Neumann boundary conditions, i.e., the pure traction problem. In this case, it is crucial to impose suitable restrictions on the external forces. In particular, as done in \cite{MPTJOTA, MPTARMA}, here we shall assume  they have null resultant and null momentum with respect to the origin, namely
\begin{equation*}
\label{globalequi}
\mathbb E( \v)=0\quad\Rightarrow\quad\int_{\partial\Omega} \mathbf g\cdot\mathbf v\,d\H^{2}(x)+\int_{\Omega} \mathbf f\cdot\mathbf v\,dx\ =\ 0,
\end{equation*}
and that 
they satisfy  the following strict compatibility condition \KKK
\[
\lab{comp}
\int_{\partial\Omega}\mathbf g\cdot\mathbf W^{2}\,x\ d\H^{2}(x)\,+\,\int_{\Omega}\mathbf f\cdot\mathbf W^{2}\,x\ dx \ <\ 0\qquad\ \forall\ \mathbf W\in \mathbf R^{3\times 3}_{\mathrm{skew}},\,\ \mathbf W\not= \mathbf 0,
\]
where $\R^{3\times 3}_{\rm skew}$ denotes the set of real $3\times3$ skew-symmetric matrices.
 For suitable classes  of  external forces, the latter condition can be interpreted as an overall dilation effect on the body, see Remark \ref{unbound} later on. \KKK
Even under such restrictions,
 in this case it is not possible to obtain { a sequential  $\Gamma$-convergence result   with respect to the weak convergence in $W^{1,p}(\Omega,\mathbb R^3)$  or to the weak $L^p(\Omega,\mathbb R^{3\times3})$ convergence of infinitesimal strain tensors}  (we stress that the elastic part in $\mathcal E^I_h$ is not invariant by infinitesimal rigid displacements\KKK, see also Remark \ref{lastremark}).
However, in this context we will show that $\mathcal E^I$ provides indeed an upper bound for the sequence $h^{-2}\mathcal E^I_h(\v_h)$ in the limit $h\to0$, as soon as $\mathbb E(\v_h)\wconv \mathbb E(\v)$ weakly in $L^p(\Omega,\mathbb R^{3\times3})$ and $\det(\mathbf I+h\nabla\v_h)=1$. On the other hand, due to the lack of control on skewsymmetric parts, we may only obtain a lower bound in terms of the functional
 \[ \mathcal F^I(\v):=
 \min_{\mathbf W\in\mathbb R^{3\times 3}_{\mathrm{skew}}}\int_\Omega\mathcal Q^I(x,\mathbb E(\v)-\tfrac12\,\mathbf W^2)\,dx -\int_\om\mathbf f\cdot\mathbf v\,dx-\int_{\partial\Omega}\mathbf g\cdot\mathbf v\,d\mathcal H^2(x).
 \]
Clearly, we have $\mathcal F^I\le \mathcal E^I$.
 Interestingly, this still allows to obtain \eqref{heuristic}-\eqref{heuristic2}, since it is possible to show (see Lemma \ref{linel} later on) that if  $\v\in\argmin\mathcal F^I$, then $\mathcal F^I(\v)=\mathcal E
 ^I(\v)$. Indeed, if $\v\in\argmin \mathcal F
 ^I$, then the minimization problem inside the definition of $\mathcal F^I$ is solved by $\mathbf W=0$.
 
 \bigskip
 
 A key step for obtaining to proof of \eqref{heuristic}-\eqref{heuristic2} will be  to approximate  divergence-free vector fields in terms of vector fields $\v_t$ having the property $\mathcal W^I(x,\mathbf I+t\nabla\v_t)<+\infty$, i.e., $\det(\mathbf I+t\nabla\v_t)=1$ for any $t>0$.  Following the approach of  \cite{MP}, we  define $$\v_t(x):=\frac{\mathbf y_t(x)-x}{t}$$
where $\mathbf y_t$ is the flow associated to the divergence-free field $\mathbf \v$, starting from $\mathbf y_0(x)=x\in\Omega$, and then by Reynolds transport formula we  see that $\mathbf y_t$ is a volume preserving deformation field. Indeed, for any $A\subseteq\Omega$ we have
\[
\frac{d}{dt}|\mathbf y_t(A)|=\int_A \dv\v(x)\,dx=0,
\]
 so that $\det\nabla\mathbf y_t=\det(\mathbf I+t\nabla\v_t)=1$.
In this paper, we will further develop this approach in order to obtain the desired recovery sequence and the upper bound under the above conditions on external loads.
Moreover, since we have natural boundary conditions we avoid the technical difficulties due to the need of keeping track of the boundary data through the construction of the recovery sequence. Therefore, we shall not need the strong regularity assumptions on $\partial\Omega$  that were imposed in \cite{MP}.
 
 \medskip
 
 
 \medskip


   Hyperlastic incompressible models are typically used to describe rubber-like materials such as artificial elastomers   as well as biological soft tissues  \cite{CF, HA, HO,KS, og2, og1, SO, Y}. 
   We refer to  \cite{AB, BA, Ch, MV, SHP, Y2} for many examples of strain energy densities that are used for the nonlinear descprition of the stress-strain behavior of these materials.
On the other hand, linear modeling is usually considered a good approximation in the small strain regime.
 Indeed,   the classical theory of linearized elasticity is based on the smallness assumption on deformation gradients, see for instance \cite{G, Lo, S}.  Nevertheless,  for a variational derivation, i.e., for the proof of
   \eqref{heuristic}-\eqref{heuristic2}, no a-priori smallness assumption is needed, 
  leading to a rigorous justification of linearized elasticity (we also stress  that small loads need not give rise to  small strains in rubber-like materials, due to their high compliance in shear). 
The first rigorous variational derivation of linearized elasticity from finite elasticity is given in \cite{DMPN},
 where $\Gamma$-convergence and convergence of minimizers of the associated Dirichlet boundary value problems are proven in the compressible case. We refer to \cite{ABK,ADMDS,ADMLP,CD,JS,MPTJOTA,MPTARMA,MP, Sc} for many other results of this kind, some of which including theories for incompressible materials \cite{CD, JS, MP}.  The study of asymptotic properties of minimal energies, similar to \eqref{heuristic}-\eqref{heuristic2}, is also typical of dimension reduction problems, see for instance  \cite{ABP,LM, PT1, PTplate, PT2, PT4, PPG}. 
 \KKK


\medskip

In the next section we rigorously state the main theorem, providing the proof of \eqref{heuristic}-\eqref{heuristic2} for the pure traction problem.
A related result has been recently obtained by Jesenko and Schmidt in \cite{JS} under different assumptions on the external loads, but in the more general framework of multiwell potentials that leads to a suitable quasiconvex envelope of the strain energy density in the limiting functional.  


%
\bigskip
 
 \subsection*{Plan of the paper} In Section 2 we  collect the the assumptions of the theory and we state the main result about convergence of minimizers. The latter is based on suitable compactness properties of (almost) minimizing sequences that are established in section 4, after some preliminaries in Section 3. In section 5, we provide the lower bound. Section 6 delivers the upper bound. Section 7 completes the proof of the main result. 

\section{Main result}\label{sectmain}

In this section we introduce all the assumptions and we state the main result. 
Let $\om\subset \R^3$ be the reference configuration of the body. We assume that for some $m\in\mathbb N$
\begin{equation}\label{OMEGA} \begin{array}{ll}&\om \hbox{ is a bounded open connected Lipschitz set, } \partial\Omega \hbox{ has $m$ connected components}.
\end{array}
\end{equation}

\subsection*{Assumptions on the elastic energy density} 

We let
  $\W^I : \om \times \mathbb R^{3 \times 3} \to [0, +\infty ]$ be  ${\mathcal L}^3\! \times\! {\mathcal B}^{9} $- measurable. 
The assumptions on $\W^I$ are similar to the ones in \cite{ADMDS,MP}, i.e., for a.e. $x\in\Omega$
%
\[
\tag{$\bb{\mathcal W0}$}\label{infty} \mathcal W^I(x,\mathbf F)=+\infty\qquad\mbox{if $\det\mathbf F\neq 1$},
\]
\beeq \lab{framind}\tag{$\bb{\mathcal W1}$} \W^I(x, \mathbf R\mathbf F)=\W^I(x, \mathbf F) \qquad \forall \, \mathbf\! \mathbf R\!\in\! SO(3), \quad \forall\, \mathbf F\in \mathbb R^{3 \times 3},
\eneq
\beeq \lab{Z1}\tag{$\bb{\mathcal W2}$}
\min \mathcal W^I=\KKK	\W^I(x,\mathbf I)=0.
\eneq
Moreover, we define $\mathcal W:\Omega\times\mathbb R^{3\times3}\to[0,+\infty]$ by $$\mathcal W(x,\mathbf F):=\W^I(x,(\det\mathbf F)^{-1/3}\mathbf F)$$ and we assume  that $\W(x,\mathbf F)$ is $C^2$ in a neighbor of rotations, i.e., 
\beeq\begin{array}{ll}\lab{reg}\tag{$\bb{\mathcal W3}$} &   \hbox{there exists a neighborhood} \ \mathcal U \  \hbox{of}\ SO(3) \hbox{ s.t., for a.e. $x\in\om$,  }
\mathcal W(x,\cdot)\in C^{2}(\mathcal U),\\
& \hbox{with a modulus of continuity of $D^2\mathcal W(x,\cdot)$ that does not depend on $x$}.\\
&\hbox{Moreover,  there exists } K>0\ \hbox{such that} \ |D^2 \mathcal W(x,\mathbf I)|\le K\,\,\,\hbox{for a.e. $ x\in\om$}.
\end{array}
\eneq
We assume the following growth property from below: there exist $C>0$ and $p\in(1,2]$ such that for a.e. $x\in\Omega$
\beeq \lab{coerc}\tag{$\bb{\mathcal W4}$}
\begin{array}{ll}
\W^I(x,\mathbf F)\ge  C \ g_{p}(d(\mathbf F, SO(3)))\qquad
 \forall\, \mathbf F\in \mathbb R^{3 \times 3},
\end{array}
\eneq
where $g_p:[0,+\infty)\to\mathbb R$ is the convex function defined by
\beeq\lab{gp}
g_{p}(t)=\left\{\begin{array}{ll} \!\! t^{2}\quad &\hbox{if}\ 0\le t\le 1\vspace{0.3cm}\\
\!\! \displaystyle \frac{2t^{p}}{p}-\frac{2}{p}+1\quad &\hbox{if}\ t\ge 1.\\
\end{array}\right.
\eneq
Here,  $SO(3)$ denotes the group of positive rotation matrices.  $d(\cdot,SO(3))$ denotes the distance from $SO(3)$ and it satisfies
\begin{equation}\label{dSO3}
d
(\mathbf F,SO(3)):=\inf_{\mathbf R\in SO(3)}|\mathbf F-\mathbf R|= |\sqrt{\mathbf F^T\mathbf F}-\mathbf I |\qquad \mbox{ for any  $\mathbf F\in\mathbb R^{3\times3}$ with $\det\mathbf F>0$},
\end{equation}
 where $|\cdot|$ is the Euclidean norm on $\mathbb R^{3\times3}$, i.e., $|\mathbf F|:=\sqrt{\mathrm{Tr}(\mathbf F^T\mathbf F)}$.

\KKK


%



\subsection*{Assumptions on the external forces}
We introduce  a  body force field $\mathbf f\in L^{\frac{3p}{4p-3}}(\Omega,\R^3)$ and a surface force field $\mathbf g\in L^{\frac{2p}{3p-3}}(\partial\Omega,\R^3)$, where $p$ is such that \eqref{coerc} holds. The corresponding contribution to the energy is given by the linear functional \begin{equation*}\label{external}
\mathcal L(\v):=\int_\om\mathbf f\cdot\v\,dx+\int_{\partial\om}\mathbf g\cdot\v\,d\mathcal H^2(x),\qquad \v\in W^{1,p}(\om,\mathbb R^3).
\end{equation*}
We note that since $\Omega$ is a bounded Lipschitz domain, the Sobolev embedding $W^{1,p}(\Omega,\mathbb R^3)\hookrightarrow L^{\frac{3p}{3-p}}(\Omega,\mathbb R^3)$ and the Sobolev trace embedding $W^{1,p}(\Omega,\mathbb R^3)\hookrightarrow L^{\frac{2p}{3-p}}(\partial\Omega,\mathbb R^3)$ imply that $\mathcal L$ is a bounded functional over $W^{1,p}(\Omega,\mathbb R^3)$.

 We assume that external {  loads have null resultant and null momentum, i.e.,
 \begin{equation}\lab{globalequi}\tag{$\bb{\mathcal L1}$}
\mathbb E(\v)=0\;\;\Rightarrow\;\;\mathcal L(\v)=0,
\end{equation}
 and that they satisfy the following strict compatibility condition}
\begin{equation}\lab{comp}\tag{$\bb{\mathcal L2}$}
\mathcal L(\mathbf W^2x)<0\qquad \forall\,\mathbf W\in \mathbb R^{3\times 3}_{\rm skew}, \ \mathbf W\neq \mathbf 0.
\end{equation}
Some examples of external loads satisfying the above assumptions are provided in the remarks at the end of this section.
\KKK

\subsection*{Statement of the main result}
 The functional representing the scaled total energy is denoted by $\mathcal F_h^{I}: W^{1,p}(\Omega,\R^3)\to \R\cup\{+\infty\} $ and defined as follows
\begin{equation*}
\label{nonlinear}
\displaystyle \mathcal F_h^{I}(\v):=\frac1{h^2}\int_\om\mathcal W^I(x,\mathbf I+h\nabla\v)\,dx
-\mathcal L(\v).
\end{equation*}
%
We further
introduce  the functional of linearized incompressible elasticity $\mathcal E^I:W^{1,p}(\Omega,\mathbb R^3)\to\mathbb R\cup\{+\infty\}$ as
\begin{equation*}\lab{elfunc}
\displaystyle {\mathcal E}^I(\v):=\left\{\begin{array}{ll}\displaystyle\int_\om \mathcal Q^I(x, \mathbb E(\v))\,dx-\mathcal L(\v)\quad &\hbox{if} \ \v\in H^1(\om,\mathbb R^3)\\
&\\
 \ \!\!+\infty\quad &\hbox{otherwise in} \ W^{1,p}(\om,\mathbb R^3),
\end{array}\right.
\end{equation*}
where 
\begin{equation*}\lab{V0}
{\mathcal Q}^I (x, \mathbf B):=\left\{\begin{array}{ll} &\dfrac12\,\mathbf B^T\,D^2\mathcal W(x,\mathbf I)\,\mathbf B\ \qquad \hbox{if}\ \ \tr\,\mathbf B=0\\
&\\
& +\infty\ \hbox{otherwise,}
\end{array}\right.
\end{equation*}
and  the limit energy functional $\mathcal F^{I}: W^{1,p}(\Omega,\R^3)\to \R\cup \{+\infty\}$ defined by
\begin{equation}
\label{DTfunc}
\F^{I}(\v): =\left\{\begin{array}{ll}
\displaystyle
\inf_{\mathbf W\in \mathbb R^{3\times 3}_{\mathrm{skew}}}\int_\om \mathcal Q^I\left (x,\mathbb E(\v)-\frac{\mathbf W^2}{2}\right )\,dx
 -\mathcal L(\mathbf v)\; &\hbox{if} \ \v\in H^1(\om,\mathbb R^3)\\
 &\\
 \ \!\!+\infty\; &\hbox{otherwise in} \ W^{1,p}(\om,\mathbb R^3).
 \end{array}
 \right.
\end{equation}
We notice that $\mathcal E^I(\v)$ is finite only if $\v\in H^1(\om,\mathbb R^3)$ is divergence-free, while $\mathcal F^I(\v)$ is finite  only if $\v$ has constant nonpositive divergence (since $\mathbf W^2$ is negative semi-definite  for any $\mathbf W\in\R^{3\times3}_{\rm skew}$). 
\KKK



We are ready for the statement of the main result
\begin{theorem}
\label{mainth1}
Assume \eqref{OMEGA},\eqref{globalequi},\eqref{comp}, \eqref{infty}, \eqref{framind}, \eqref{Z1}, \eqref{reg},    
\eqref{coerc}. 
Then for every vanishing sequence $(h_j)_{j\in\mathbb N}\subset(0,1)$  
we have
\beeq\lab{convmin}
\inf_{W^{1,p}(\om,\mathbb R^3)}\mathcal F^{I}_{h_{j}} \in \mathbb R. \eneq
If $(\mathbf v_j)_{j\in\mathbb N}\subset W^{1,p}(\om,\mathbb R^3)$ is a sequence  such that
\begin{equation} \label{assinf2}  \lim_{j\to +\infty}\left( \mathcal F^{I}_{h_{j}}(\v_{j})-\inf_{W^{1,p}(\om,\mathbb R^3)}\mathcal F^{I}_{h_{j}}\right)= 0,
\end{equation}
then there is a (not relabeled) subsequence such that
\[ \mathbb E(\v_{j})\wconv \mathbb E(\v_*) \ \hbox{weakly in}\ L^{p}(\Omega,\R^3)\qquad\mbox{as $j\to+\infty$},
\]
where $\v_*\in H^1(\om,\R^3)$ is a minimizer of $\mathcal F^I$ over $W^{1,p}(\om,\mathbb R^3)$,
and
\begin{equation*}   \mathcal F^{I}_{h_{j}}(\v_{j})\to \mathcal F^{I}(\v_*)={\mathcal E}^I(\v_*),\quad
\inf_{W^{1,p}(\om,\R^3)}\mathcal F^{I}_{h_{j}}\to \min_{W^{1,p}(\om,\R^3)}\mathcal F^{I}= \min_{W^{1,p}(\om,\R^3)}\mathcal E^{I} \quad\mbox{as $j\to+\infty$}.
\end{equation*}
\end{theorem}

We close this section with several remarks about the main theorem.

\begin{remark}\rm
It is worth noticing that the infimum in the right hand side of \eqref{DTfunc} is actually a minimum. Indeed if $\v\in H^1(\om,\mathbb R^3)$ then either $\F^{I}(\v)=+\infty$ or $\dv\v$ is a non positive constant. In the latter case let $(\mathbf W_n)_{n\in\mathbb N}\subset \mathbb R^{3\times 3}_{\rm skew}$ be a minimizing sequence: then $|\mathbf W_n|^2=\tr (\mathbf W_n^T\mathbf W_n)=-\tr \mathbf W_n^2=-2\dv\v$ hence there exists $\mathbf W_{\v}$ such that, up to subsequences, $\mathbf W_n\to \mathbf W_{\v}$ and 
$\F^{I}(\v)=\int_\om \mathcal Q^I\left (\mathbb E(\v)-\frac{\mathbf W_\v^2}{2}\right )\,dx-\mathcal L(\v)$ as claimed.
\end{remark}
\begin{remark}\rm If the function $\mathcal W$ is replaced by any other $\widetilde {\mathcal W}$, still satisfying 
  \eqref{reg},    
such that $\mathcal W(x,\mathbf F)=\widetilde {\mathcal W}(x,\mathbf F)$ if $\det \mathbf F=1$, then  this does not affect functionals $\mathcal F^I$ and $\mathcal E^I$. Indeed, if $\mathrm{Tr}\,\mathbf B=0$ then $\det(\exp(h\mathbf B))=\exp(h\tr \mathbf B)=1$ and by Taylor's expansion we have for a.e. $x\in\Omega$
\[
\lim_{h\to 0}\frac1{h^2}\mathcal W^I(x,\exp(h\mathbf B))=\frac12\mathrm{sym}\mathbf B\,D^2\widetilde{\mathcal W}(x,\mathbf I)\,\mathrm{sym}\mathbf B=  \frac12\mathrm{sym}\mathbf B\,D^2\mathcal W(x,\mathbf I)\,\mathrm{sym}\mathbf B=\mathcal Q^I(x,\mathbf B).
\] 
\end{remark}

\begin{remark}\rm
A typical form of  $\mathcal W^I$  is the  Ogden incompressible strain energy density, see  \cite{C, HA, OGDEN}, 
given by
\begin{equation*}\lab{ogden}
\mathcal W^I(\mathbf F):=\sum_{k=1}^{N}\frac{\mu_k}{\alpha_k}(\mathrm{Tr}((\mathbf F^T\mathbf F)^{\alpha_k/2})-3),\qquad \det\mathbf F=1,
\end{equation*}
where  $N, \mu_k, \alpha_k$ are  material constants, and extended to $+\infty$ if $\det\mathbf F\neq1$. If the material constants vary in a suitable range, the Ogden model satisfies the assumptions \eqref{framind}, \eqref{Z1}, \eqref{reg},    
\eqref{coerc}.  
In particular, we  refer to \cite{ADMDS} and \cite{MP} for a discussion about the growth properties and   the validity of \eqref{coerc} for the Ogden strain energy density and other standard models.
\end{remark}
 
 \begin{remark}\label{rmk2.2}
\rm It is worth to stress that 
 Theorem \ref{mainth1} does not hold
 without suitable compatibility assumptions on external forces. Not even relaxing \eqref{comp} by requiring non strict inequality therein would work.   \KKK
Indeed, choose $\mathbf f=\mathbf g\equiv 0$,
\begin{equation}
\lab{Wquad}
\W^I(\xx,\mathbf F)=\left\{\begin{array}{ll}
 &|\mathbf F^{T}\mathbf F-\Id|^{2}\  \hbox{if} \ \det \mathbf F=1\,,
 \vspace{0.1cm}\\
  &+\infty\ \hbox{ otherwise\,,}
 \end{array}\right.
\end{equation}
so that \eqref{infty}, \eqref{framind}, \eqref{Z1}, \eqref{reg}, \eqref{coerc} are satisfied (with $p=2$). 
Let
$
\v_j:=h_{j}^{-1}(\mathbf R-\mathbf I)x
$,
where $\mathbf R\in SO(3),\ \mathbf R\not=\mathbf I$ and $(h_j)_{j\in\mathbb N}\subset(0,1)$ is a vanishing sequence.
Then $\mathbf y_j=x+h_j\v_j=\mathbf Rx$ hence $\det\nabla\mathbf y_j=1$ and $\ \mathcal F_{h_{j}}^I(\v_j)= 0=\min_{W^{1,p}(\om,\mathbb R^{3})} \mathcal F_{h_{j}}^I$ so the sequence $(\v_j)_{j\in\mathbb N}$ satisfies \eqref{assinf2}. However, it has no subsequence that is
weakly converging in  $W^{1,p}(\om,\mathbb R^{3})$.  Moreover, $\mathbb E(\v_j)=c\,h_j^{-1}\,\mathbf W^2$ for some $c\in(0,1]$ and some $\mathbf W\in\mathbb R^{3\times3}_{\rm skew}$ such that $|\mathbf W|^2=2$, as a consequence of the Euler-Rodrigues formula \ref{eurod} for the representation of rotations that we shall recall in Section \ref{prel}. Therefore, the sequence $(\mathbb E(\v_j))_{j\in\mathbb N}$ has no subsequence that is bounded in $L^p(\om,\R^{3\times3})$. \KKK
\end{remark}

\begin{remark}\label{lastremark}\rm Under the assumptions of Theorem \ref{mainth1}, in general it is not possible to get  weak $W^{1,p}(\om,\R^3)$ compactness  of (almost) minimizing sequences.  Indeed, let us consider the following example. Let $\Omega=B_1$ be the unit ball of $\mathbb R^3$, centered at the origin. Let $\mathcal W^I$ be given by \eqref{Wquad}. Let $\mathbf f(x)=x$ and $\mathbf g\equiv 0$. It is readily seen that \eqref{globalequi} and \eqref{comp} are satisfied. On the other hand, let the divergence-free vector field $\v_*\in H^1(\om,\mathbb R^3)$ be a minimizer of $\mathcal E^I$ over $W^{1,p}(\om,\R
^3)$. Since $\om=B_1$ and $\v_*$ is divergence-free, there exists $\w_*\in H^2(\om,\mathbb R^3)$ such that $\v_*=\mathrm{curl}\, \w_*$ and by divergence theorem
\[
\mathcal L(\v_*)=-\int_{B_1}x\cdot\v_*\,dx=-\int_{B_1}x\cdot\mathrm{curl}\,\w_*\,dx=\int_{B_1}\dv(x\wedge\w_*)\,dx=\int_{\partial B_1}(x\wedge\w_*)\cdot x\,dx=0
\] 
so that 
$$\mathcal E^{I}(\v_*)=4\int_{B_1} |\mathbb E(\v_*)|^2\,dx \ge 0,$$
and since
 $\mathcal E^I(\mathbf 0)=0$, we get $\mathcal E^I(\v_*)=0$. By Theorem \ref{mainth1}, we deduce that $0$ is the minimal value of both $\mathcal E^I$ and $\mathcal F^I$ over $W^{1,p}(\om,\mathbb R^3)$. Let now $\mathbf W\in\mathbb R^{3\times3}_{\rm skew}$ be such that $|\mathbf W|^2=2$. Let $\alpha \in(\tfrac12,1)$ and let us consider a vanishing sequence $(h_j)_{j\in\mathbb N}\subset(0,1)$ and the sequence 
 \begin{equation}\label{special}
 \v_{j}(x):=h_j^{\alpha-1}\mathbf Wx+h_j^{-1}\left(1-\sqrt{1-h_j^
 {2\alpha}}\right)\mathbf W^2x=0.
 \end{equation}
 The Euler-Rodrigues formula \eqref{eurod} implies that for any $j\in\mathbb N$ there holds  $\mathbf y_{j}(x):=x+h_j\v_{j}=\mathbf R_jx$ for a suitable $\mathbf R_j\in SO(3)$. This implies $\det\nabla \mathbf y_j=1$ for any $j\in\mathbb N$ and then 
by \eqref{Wquad} we get 
 \[
 \mathcal F^I_{h_j}(\v_j)=\frac1{h_j^2}\int_{B_1}\mathcal W(x, \mathbf I+h_j\nabla\v_j)\,dx=\frac1{h_j^2}\int_{B_1}|2h_j\mathbb E(\v_j)+h_j^2\nabla\v_j^T\nabla\v_j|^2\,dx.
 \]
 The above right hand side goes to $0$ as $j\to+\infty$, as shown after a computation  making use of \eqref{special}.  Therefore the sequence $(\v_j)_{j\in\mathbb N}$ satisfies \eqref{assinf2}. On the other hand, 
  $(\nabla\v_j)_{j\in\mathbb N}$ has no subsequence that is bounded in $L^p(\om,\mathbb R^{3\times3})$, since $\alpha\in(\tfrac12,1)$. 
\end{remark}
\begin{remark}\lab{unbound}
\rm  If $\mathcal L(\mathbf W_*^2 x) >0$ for some $\mathbf W_*\in \R^{3\times 3}_{\rm skew}$ then the  functionals $\mathcal F^I_h$ may admit no uniform bound from below as $h\to 0$. Indeed, let $\Omega=\{(x_1,x_2,x_3)\in\mathbb R^3: x_1^2+x_2^2<1,\,0<x_3<1\}$ and let 
$\mathbf g(x)=(-x_1,-x_2,0),\ \ \mathbf f\equiv 0,  \ \mathbf W_*:= \mathbf e_1\otimes\mathbf e_2-\mathbf e_2\otimes\mathbf e_1$. Then \eqref{globalequi} is satisfied and since ${\mathbf W}_*^2= -\mathbf e_1\otimes\mathbf e_1-\mathbf e_2\otimes\mathbf e_2$ we get 
\[\mathcal L({\mathbf W}_*^2 x)=\int_{\partial\om}(x_1^2+x_2^2)\,d\H^2(x) >0\]
and by taking a vanishing sequence $(h_j)_{j\in\mathbb N}\subset(0,1)$ and $\v_j:= h_j^{-1}(\mathbf W_* x+{\mathbf W}_*^2 x)$ a direct computation shows that
\[\mathcal F^I_{h_{j}}(\v_j)=- h_j^{-1}\mathcal L({\mathbf W}_*^2 x)\to -\infty\]
as $j\to+\infty$, as claimed.

For  general $\Omega$, we notice that if the body is subject to a uniform boundary compressive  force field  then the above situation occurs.   Indeed, if $\n$ denotes the outer unit normal vector to $\partial\om$, and we choose
$\mathbf g=\lambda\n$ with $\lambda< 0$ and $\mathbf f\equiv 0$, then
\[
\int_{\partial\Omega}\mathbf g\cdot\mathbf W^{2} x\ d\H^{2}(x)\ =\ \lambda\,(\Tr\mathbf W^{2})\,\vert\Omega\vert
\ >\ 0
\qquad \forall \
\mathbf W\!\in\! \R^{3\times 3}_{\rm skew},\ \mathbf W\neq 0
\]
and by choosing  $\mathbf W\!\in\! \R^{3\times 3}_{\rm skew},\ \mathbf W\neq 0$ and $\v_j:= h_j^{-1}(\mathbf Wx+{\mathbf W}^2 x)$ we get
\[\mathcal F^I_{h_{j}}(\v_j)= \mathcal L (\v_j)=-\frac{\lambda}{h_{j}}\int_{\partial \Omega}\mathbf W^{2}x\cdot\n\,d\H^{2}(x)=-\frac{\lambda}{h_{j}}(\Tr\mathbf W^{2})|\Omega|\rightarrow -\infty\]
as $j\to+\infty$, as before.  On the other hand, if $\lambda>0$ we have a dilation  effect on the body and \eqref{comp} is satisfied.  \KKK
\end{remark} 
\begin{remark}\label{phi} \rm
 Let us consider external forces of the following form. Given $p$ such that \eqref{coerc} holds, let  $\mathbf f=\nabla\phi$, where $\phi\in W^{1,r}_0(\Omega)$, $r= \tfrac{3p}{4p-3}$, and let $\mathbf g=\lambda \mathbf n$, where $\lambda\in\R$ and $\mathbf n$ is the unit exterior normal vector to $\partial\Omega$, with  $\int_\om\phi(x)\,dx<\lambda|\om|$. 
It is readily seen that in this case \eqref{globalequi} and \eqref{comp} are satisfied.
Moreover, by the divergence theorem, $\mathcal L(\v)=0$ for every divergence-free vector field $\v\in H^1(\om,\mathbb R^3)$. Therefore, under the assumptions of Theorem \ref{mainth1}, from the definition of $\mathcal E^I$ and from the estimate \eqref{ellipticity} below we deduce that $\mathrm{argmin}_{W^{1,p}(\om,\mathbb R^3)}\mathcal E^I$ coincides with the set of rigid displacements of $\Omega$ (i.e., displacements fields with vanishing infinitesimal strain tensor).
From Theorem \ref{mainth1} we deduce in this case that the minimal value of both $\mathcal E^I$ and $\mathcal F^I$ is $0$.

\end{remark}
 \section{Notation and preliminary results}\label{prel}
Through the paper, $\mathbb R^{3\times 3}$ will denote the set of $3\times 3$ real matrices.
$\mathbb R^{3\times 3}_{\rm sym}$ and $\mathbb R^{3\times 3}_{\rm skew}$ denote respectively the sets of symmetric and skew-symmetric matrices and for every $\mathbf B\in \mathbb R^{3\times 3}$ we define ${\rm sym\,}\mathbf B:=\frac{1}{2}(\mathbf B+\mathbf B^T)$
and  ${\rm skew\,}\mathbf B:=\frac{1}{2}(\mathbf B-\mathbf B^T)$. Moreover,   we set
\begin{equation}\label{cone}
\mathbb K:=\{\tau(\mathbf R-\mathbf I): \tau \geq 0,\  \mathbf R\in SO(3)\}\,.
\end{equation}
Given $\mathbf a, \,\mathbf b\in\R^3$, with $\mathbf a\wedge\mathbf b$ we denote the cross product.
A Sobolev vector field $\w\in W^{1,1}(\om,\mathbb R^3)$ is said to be an infinitesimal rigid displacement  if $\mathbb E(\v):={\rm sym\,}\nabla\v=0$ a.e. in $\om$, which is the case iff there exist $\aa,\,\mathbf b\in\R^3$ such that $\v(x)=\aa \wedge x+\mathbf b$ for
every $x\in \om$. By $H^1_{\mathrm{div}}(\om,\mathbb R^3)$ we denote the space of divergence-free $H^1(\om,\R^3)$ vector fields. The codomain of functions of $L^r(\Omega)$ or $W^{1,r}(\Omega)$ shall be
$\R$, $\R^3$ or $\R^{3\times3}$ and we shall often omit it from the notation. Bold letters will be used  for vector fields. 

 \subsection*{Euler-Rodrigues formula}
 For every $\mathbf R\in SO(3)$ there exist $\vartheta\in \mathbb R$ and $\mathbf W\in \mathbb R^{3\times 3}_{\mathrm{skew}},$ such that  $|\mathbf W|^{2}=2$ and such that $\exp(\vartheta\,\mathbf W)= \mathbf R$.
By taking into account that $\mathbf W^{3}=-\mathbf W$, the exponential matrix series $ \exp(\vartheta\,\mathbf W)=\sum_{k=0}^\infty \vartheta^k\mathbf W^k/k!$ yields the {\it Euler-Rodrigues formula}:
\begin{equation}\label{eurod}
\exp(\vartheta\,\mathbf W)\,=\,\mathbf R\,=\,\mathbf I\,+\,\sin\vartheta \,\mathbf W\,+\,(1-\cos\vartheta)\,\mathbf W^{2} .
\end{equation}
We also recall that
if  $\mathbf W\in \mathbb R^{3\times 3}_{\mathrm{skew}}$ and  $|\mathbf W|^{2}=2$ then  $|\mathbf W^2|^{2}=2$.


 \subsection*{Properties of $\mathcal W$} Let assumptions \eqref{infty}, \eqref{framind}, \eqref{Z1}, \eqref{reg} and
\eqref{coerc}  hold.
  We recall that $\mathcal W$ is defined by $\mathcal W(x,\mathbf F):=\mathcal W^I(x,(\det\mathbf F)^{-1/3}\mathbf F)$, thus $\W^I\ge \W$.
  It is clear that $\mathcal W$ itself satifies  
 \eqref{framind}  and \eqref{Z1}, so that by \eqref{reg}, since  $\mathcal W\ge 0$, we deduce 
 \begin{equation}\label{smooth0}\W(x,\mathbf R)\!=\!0,\ D \W(x,\mathbf R)\!=\!0 \quad \forall \mathbf R\in SO(3),\quad \mbox{for a.e. $x\in\om$}.\end{equation}
Due to frame indifference there exists a function $\mathcal V$ such that 
\begin{equation}\label{vi}
\W(x,\mathbf F)=\V(x,\textstyle{\frac{1}{2}}( \mathbf F^T \mathbf F - \mathbf I))\,,
\qquad
\ \forall\, \mathbf F\in \mathbb R^{3\times 3},\  \hbox{ for a.e. }x\in \om.
\end{equation}
Given $\mathbf B\in\mathbb R^{3\times 3}$,  $h> 0$, we have
$
\mathcal W(x,\Id+h\mathbf B)=\V(x,h\,{\rm sym}\mathbf B+\tfrac12h^{2}\mathbf B^{T}\mathbf B).
$
 By \eqref{smooth0} and by \eqref{reg} we get for a.e. $x\in\Omega$
\[\displaystyle 
\lim_{h\to 0} h^{-2}\mathcal W(x,\mathbf I+h\mathbf B)=
\frac{1}{2} \,{\rm sym} \mathbf  B\, D^2\V (x, \mathbf 0) \ {\rm sym}\mathbf  B=\frac12\, \mathbf B^T D^2\mathcal W(x,\mathbf I)\,\mathbf B,
\]
hence \eqref{coerc} and \eqref{dSO3} imply that for a.e. $x\in \om$, as soon as $\Tr \mathbf B=0$,
\begin{equation}\label{ellipticity}\begin{aligned}
\frac12\, \mathbf B^T D^2\mathcal W(x,\mathbf I)\,\mathbf B&=\lim_{h\to 0} h^{-2}\mathcal W(x,\mathbf I+h\mathbf B+o(h))=\lim_{h\to0} h^{-2}\mathcal W(x,\exp(h\mathbf B))\\&=\lim_{h\to0} h^{-2}\mathcal W^I(x,\exp(h\mathbf B))\ge \limsup_{h\to 0}Ch^{-2}\,d^2(\exp(h\mathbf B),SO(3))\\&=\limsup_{h\to 0} C h^{-2}\left|\sqrt{\exp(h\mathbf B)^T\exp(h\mathbf B)}-\mathbf I\right|^2\\&= \limsup_{h\to 0} Ch^{-2}|\exp(h\,\mathrm{sym}\mathbf B)-\mathbf I|^2=C|\mathrm{sym}\mathbf B|^2.\end{aligned}
\end{equation}
%
 Moreover, by expressing the remainder of Taylor's expansion in terms of the $x$-independent modulus of continuity $\omega$ of $D^2\W(x,\cdot)$ on the set $\mathcal U$ from \eqref{reg}, we have
\begin{equation}\lab{regW}
\left|\mathcal W(x, \mathbf I+h\mathbf B)- \frac{h^2}{2} \,{\rm sym} \mathbf  B\, D^2\mathcal W (x, \mathbf I) \ {\rm sym}\mathbf  B\right|\le h^2\omega(h|\mathbf B|)|\mathbf B|^2
\end{equation}
for any small enough $h$ (such that $h\mathbf B\in\mathcal U$),
where
 $\omega:\mathbb R_+\to\mathbb R$ is such that $\lim_{t\to 0^+}\omega(t)=0$. Similarly, $\mathcal V(x,\cdot)$ is $C^2$ in a neighbor of the origin in $\mathbb R^{3\times 3}$, with an $x$-independent modulus of continuity $\eta:\mathbb R_+\to \mathbb R$, which is increasing and such that $\lim_{t\to0^+}\eta(t)=0$, and  we have
\begin{equation}\lab{regV}
\left|\mathcal V(x, h\mathbf B)- \frac{h^2}{2} \,{\rm sym} \mathbf  B\, D^2\V (x, \mathbf 0) \ {\rm sym}\mathbf  B\right|\le h^2\eta(h|\mathbf B|)|\mathbf B|^2
\end{equation}
for any small enough $h$.

 \subsection*{Sobolev-Poincar\'e inequality}
  Here and for the rest of this section,  $\Omega$ is a bounded connected set with Lipschitz boundary.
  Let $p\in(1,2]$.
 By Sobolev embedding, Sobolev trace embedding and by the Poincar\'e inequality for null-mean functions
  we have for any $\v\in W^{1,p}(\Omega,\mathbb R^3)$
 \begin{equation}\label{fried}
 \|\v-\bar\v\|_{L^{\frac{3p}{3-p}}(\Omega,\mathbb R^3)}+\|\v-\bar\v\|_{L^{\frac{2p}{3-p}}(\partial\Omega,\mathbb R^3)}\le K_F\|\nabla \v\|_{L^p(\Omega,\mathbb R^{3\times3})},
 \end{equation} 
 where $K_F$ is a constant only depending on $\Omega,p$ and $\bar\v:=\frac1{|\Omega|}\int_\Omega\v\,dx$.

 \subsection*{Projection on rigid motions}
 Let $p\in(1,2]$ and let
$$\mathcal R\!:=\!\{ \vv\in W^{1,1}(\Omega,\R^3): \E(\v)=\mathbf 0\}$$ denote the space spanned by the set of the infinitesimal rigid displacements.  We denote by
$\mathbb P\vv$  the  unique   
projection  of $\vv\in W^{1,p}(\Omega,\mathbb R^3)$ onto $\mathcal R$.

 \subsection*{Korn Inequality} Let $p\in(1,2]$. 
 For any $\v\in W^{1,p}(\Omega,\mathbb R^3)$, there is a unique couple $\mathbf W_\v\in\mathbb R^{3\times3}_{\mathrm{skew}}$, $\mathbf a_{\v}\in\mathbb R^3$ such that
 \[
 \|\v-\mathbb P\v\|_{L^p(\Omega)}=\min\{\|\v-(\mathbf W x+\mathbf a)\|_{L^p(\Omega)}: \mathbf W\in\mathbb R^{3\times3}_{\mathrm{skew}},\, \mathbf a\in\mathbb R^3\}=\|\v-(\mathbf W_\v\, x+\mathbf a_\v)\|_{L^p(\Omega)}
 \]
 and Korn inequality, see for instance \cite{N}, entails the
  existence of a constant $Q_{K}=Q_K(\om,p)$ such that
  \begin{equation}\label{secondkorn}
  \|\nabla\v-\mathbf W_\v\|_{L^p(\Omega)}\le Q_K \|\mathbb E(\v)\|_{L^p(\Omega)}.
  \end{equation}
  Moreover, by combining the latter with Sobolev and trace inequalities, we obtain the existence of a further constant $C_K=C_K(\om,p)$ such that for all $\v\in W^{1,p}(\om,\R^3)$
\begin{equation}\label{kornpoi}
\|\vv-\mathbb P\vv\|_{L^{\frac{3p}{3-p}}(\Omega,\R^3)}\,+\,
\|\vv-\mathbb P\vv\|_{L^{\frac{2p}{3-p}}(\partial\Omega,\R^3)}
\,\le\ C_{K}\ \|\mathbb E(\vv)\|_{L^p(\Omega,\mathbb R^{3 \times 3})}.
\end{equation}
 

 \subsection*{Basic estimate on  external forces}
As a consequence of \eqref{kornpoi}, if  \eqref{globalequi} holds true we obtain the following estimate for functional $\mathcal L$: for any $\v\in W^{1,p}(\Omega,\mathbb R^3)$,  there holds 
\begin{equation}\label{elle}\begin{aligned}
|\mathcal L(\v)|&=|\mathcal L(\v-\mathbb P\v)|\\&\le \|\mathbf f\|_{L^{\frac{3p}{4p-3}}(\Omega,\mathbb R^3)}\|\v-\mathbb P\v\|_{L^{\frac{3p}{3-p}}(\Omega,\mathbb R^3)} + \|\mathbf g\|_{L^{\frac{2p}{3p-3}}(\partial\Omega,\mathbb R^3)}\|\v-\mathbb P\v\|_{L^{\frac{2p}{3-p}}(\partial\Omega,\mathbb R^3)} \\
&\le C_{\mathcal L}\|\mathbb E(\v)\|_{L^p(\Omega,\R^{3\times3})},
\end{aligned}\end{equation}
 where $C_{\mathcal L}:=C_K\left(\|\mathbf f\|_{L^{\frac{3p}{4p-3}}(\Omega,\mathbb R^3)}+ \|\mathbf g\|_{L^{\frac{2p}{3p-3}}(\partial\Omega,\mathbb R^3)}\right)$ and $C_K$ is the constant in \eqref{kornpoi}.

 \subsection*{Rigidity inequality}
We recall the rigidity inequality by Friesecke, James and M\"uller \cite{FJM0}, in its version from  \cite{FJM}, \cite{ADMDS}.
Let $g_{p}$ the function defined in \eqref{gp}. There exists a constant $C_p=C_{p}(\om) >0$ such that for every $\yy\in W^{1,p}(\om,\mathbb R^{3})$
there exists a constant $\mathbf R\in SO(3)$ such that we have
\begin{equation}\label{muller}
\int_{\om}g_{p}(|\nabla\yy-\mathbf R|)\,dx\le C_{p}\int_{\om}g_{p}(d(\nabla \yy, SO(3)))\,dx.
\end{equation}

We close this section with a result about convergence of infinitesimal strain tensors.
\begin{lemma}
\label{corcurl}
Let $p\in(1,2]$. Let
$(\w_{n})_{n\in\mathbb N}\subset W^{1,p}(\Omega,\R^3)$ be a sequence such that $\mathbb E(\w_n)\wconv \mathbf T$ weakly in $L^p(\Omega,\mathbb R^{3\times 3})$ as $n\to+\infty$.
Then there exists $\w\in W^{1,p}(\Omega,\R^3)$ such that $\mathbf T=\mathbb E(\w)$.
 If in addition we assume that  $\nabla\w_{n}\wconv \mathbf G$ weakly in $L^{p}(\om;\mathbb R^{3\times3})$,
then there exists a constant matrix $\mathbf W\in \mathbb R^{3\times 3}_{\mathrm{skew}}$ such that
$\nabla\w=\mathbf G-\mathbf W$.
\end{lemma}
\begin{proof}
The proof is given in \cite[Lemma 3.2]{MPTARMA} for the case $p=2$. Its extension to $p\in(1,2)$ is straightforward.
\end{proof}

\section{Compactness}

We prove uniform $L^p(\Omega,\mathbb R^{3\times3})$ bounds for $\mathbb E(\v_j)$ on almost minimizing sequences $(\v_j)_{j\in\mathbb N}$ of $\mathcal F^I_{h_{j}}(\v_{j})$ as $h_j\to 0$. We start by showing that functionals $\mathcal F^I_h$ are uniformly bounded from below.

\begin{lemma}\label{lemmabound}{\bf (Boundedness from below)}. Assume \eqref{OMEGA}, \eqref{infty}, \eqref{framind},\eqref{Z1},\eqref{reg},\eqref{coerc},
   \eqref{globalequi} and \eqref{comp}.
There exists a constant $C> 0$ (only depending on $\Omega,p,\mathbf f,\mathbf g$) such that
$\mathcal F^I_h(\v)\ge-C$ for any $h\in (0,1)$ and any $\v\in W^{1,p}(\Omega,\mathbb R^3)$.
\end{lemma}
\begin{proof}
Let $\v\in W^{1,p}(\Omega,\mathbb R^3)$ and let $h\in(0,1)$. Let $\mathbf y:=\bb i+h\v$ and let $\mathbf R\in SO(3)$ be a constant matrix such that \eqref{muller} holds. Let $S:=\{x\in\Omega:|\nabla \mathbf y(x)-\mathbf R|\le 1\}$. By taking advantage of assumption \ref{coerc} and of the linearity of $\mathcal L$, since $g_p(t)=t^2$ for $0\le t\le 1$ and $g_p(t)\ge t^p$ for $t\ge 1$, we get
\begin{equation}\label{eins}\begin{aligned}
\mathcal F^I_h(\v)&\ge \frac{c}{h^2}\int_\Omega g_p(|\nabla \mathbf y-\mathbf R|)\,dx-\frac1h\,\mathcal L(\mathbf y-\bb i)\\&\ge \frac c{h^2}\, \int_S |\nabla \mathbf y-\mathbf R|^2\,dx+\frac c{h^2} \int_{\Omega\setminus S} |\nabla \mathbf y-\mathbf R|^p\,dx-\frac1h\,\mathcal L(\mathbf y-\bb i),
\end{aligned}\end{equation}
where $c>0$ is a constant only depending on $p$ and $\Omega$.
By the Sobolev-Poincar\'e inequality \eqref{fried}, letting $\mathbf u(x):=\mathbf R x$ and letting $\mathbf m$ denote the mean value of $\mathbf y-\mathbf u$ on $\Omega$, we have
\begin{equation}\label{poi}
\|\mathbf y-\mathbf u-\mathbf m\|_{L^{\frac{3p}{3-p}}(\Omega,\mathbb R^3)}
+\|\mathbf y-\mathbf u-\mathbf m\|_{L^{\frac{2p}{3-p}}(\partial\Omega,\mathbb R^3)}\le K_F\|\nabla\mathbf y-\mathbf R\|_{L^p(\Omega,\mathbb R^{3\times 3})}.
\end{equation}
By the Euler-Rodrigues formula \eqref{eurod} we represent $\mathbf R$ as $\mathbf R=\mathbf I+\sin\theta\mathbf W+(1-\cos\theta)\mathbf W^2$ for a suitable skew-symmetric matrix $\mathbf W$ and some $\theta\in(-\pi,\pi]$. Then, we notice that \eqref{globalequi} and \eqref{comp} entail $\mathcal L(\mathbf u-\bb i)\le 0$ so that $\mathcal L(\mathbf y-\bb i)\le \mathcal L(\mathbf y-\mathbf u-\mathbf m)$. As a consequence,  by means of the H\"older inequality and of  \eqref{poi} we get 
\[\begin{aligned}
&\mathcal L(\mathbf y-\bb i)\le \mathcal L(\mathbf y-\mathbf u-\mathbf m)\\&\quad\le \|\mathbf f\|_{L^{\frac{3p}{4p-3}}(\Omega,\mathbb R^3)}\,\|\mathbf y-\mathbf u-\mathbf m\|_{L^{\frac{3p}{3-p}}(\Omega,\mathbb R^3)}+\|\mathbf g\|_{L^{\frac{2p}{3p-3}}(\partial\Omega,\mathbb R^3)}\, \|\mathbf y-\mathbf u-\mathbf m\|_{L^{\frac{2p}{3-p}}(\partial\Omega, \mathbb R^3)}\\&\quad
\le C_{\mathbf f,\mathbf g}\, \|\nabla \mathbf y-\mathbf R\|_{L^p(\Omega,\mathbb R^{3\times 3})}
\le C_{\mathbf f,\mathbf g} \|\nabla\mathbf y-\mathbf R\|_{L^p(\Omega\setminus S,\mathbb R^{3\times 3})}+ C_{\mathbf f, \mathbf g} |\Omega|^{\frac{2-p}{2p}}\|\nabla\mathbf y-\mathbf R\|_{L^2(S,\mathbb R^{3\times 3})}\\
\end{aligned}
\] 
where $C_{\mathbf f,\mathbf g}:=K_F\left(\|\mathbf f\|_{L^{\frac{3p}{4p-3}}(\Omega,\mathbb R^3)}+\|\mathbf g\|_{L^{\frac{p}{p-1}}(\partial\Omega,\mathbb R^3)}\right)$, and then by Young inequality we obtain
\begin{equation}\label{zwei}\begin{aligned}
\mathcal L(\mathbf y-\bb i)&\le \frac{p-1}{p}\, C_{\mathbf f,\mathbf g}^{\frac{p}{p-1}}\left(\frac{2h}{cp}\right)^{\frac{1}{p-1}}+\frac{c}{2h}\|\nabla\mathbf y-\mathbf R\|_{L^p(\Omega\setminus S,\mathbb R^{3\times 3})}^p\\
&\qquad+ \frac{h}{2c}\,C_{\mathbf f,\mathbf g}^2\,|\Omega|^{\frac{2-p}{p}}+\frac{c}{2h}\|\nabla\mathbf y-\mathbf R\|^2_{L^2(S,\mathbb R^{3\times 3})},
\end{aligned}\end{equation}
where $c$ is the constant appearing in \eqref{eins}.
By joining together \eqref{eins} and \eqref{zwei} we get
\[\begin{aligned}
\mathcal F^I_h(\v)&\ge \frac{c}{2h^2}\|\nabla\mathbf y-\mathbf R\|_{L^p(\Omega\setminus S,\mathbb R^{3\times 3})}^p+  \frac{c}{2h^2}\|\nabla\mathbf y-\mathbf R\|^2_{L^2(S,\mathbb R^{3\times 3})}
\\&\qquad-C_{\mathbf f,\mathbf g}^{\frac{p}{p-1}}\left(\frac{2}{cp}\right)^{\frac{1}{p-1}}\,h^{\frac{2-p}{p-1}}
-\frac{1}{2c}\,C_{\mathbf f,\mathbf g}^2\,|\Omega|^{\frac{2-p}{p}} \ge
-C_{\mathbf f,\mathbf g}^{\frac{p}{p-1}}\left(\frac{2}{cp}\right)^{\frac{1}{p-1}}
-\frac{1}{2c}\,C_{\mathbf f,\mathbf g}^2\,|\Omega|^{\frac{2-p}{p}} 
\end{aligned}\]
as desired.
\end{proof}

\begin{lemma}\label{compactness}{\bf (Compactness)}.
Assume \eqref{OMEGA}, \eqref{infty}, \eqref{framind},\eqref{Z1},\eqref{reg},\eqref{coerc},
   \eqref{globalequi} and \eqref{comp}.
 Let $(h_{j})_{j\subset\mathbb N}\subset(0,1)$ be a vanishing a sequence  
 and let $(\v_{{j}})_{j\in\mathbb N}\subset W^{1,p}(\Omega,\R^3)$ be a sequence  such that \begin{equation}\label{quasi}\lim_{j\to+\infty}\left(\mathcal F^I_{h_{j}}(\v_{j})-\inf_{W^{1,p}(\Omega)}\mathcal F
 ^I_{h_j}\right) =0.\end{equation} 
  Then there exists $M>0$ such that $\|\mathbb E(\v_{j})\|_{L^{p}(\Omega)}\le M$ for any $j\in\mathbb N$.
\end{lemma}
\begin{proof}
The argument extends the one of \cite[Lemma 3.6]{MPTARMA} to the weaker coercivity condition \eqref{coerc}.

For any $j\in\mathbb N$, by Lemma \ref{lemmabound} there holds \begin{equation}\label{ser}
-\infty<
 \inf_{W^{1,p}(\Omega)}\mathcal F^I_{h_j}\le \mathcal F^I_{h_j}(\bb 0)=0,\end{equation} therefore by considering \eqref{quasi} it is not restrictive to assume that $\mathcal F^I_{h_j}(\v_j)\le 1$ for any $j\in\mathbb N$. We assume by contradiction that the sequence $(t_j)_{j\in\mathbb N}$, defined as $t_j:=\|\mathbb E(\v_j)\|_{L^p(\Omega)}$, is unbounded, so that up to extraction of a not relabeled subsequence we have $t_j\to+\infty$ as $j\to+\infty$ and moreover $t_jh_j$ converge to a limit as $j\to+\infty$. We let $\w_j=\v_j/t_j$, so that $\|\mathbb E(\w_j)\|_{L^p(\Omega)}=1$ for any $j\in\mathbb N$, and along a not relabeled subsequence we have $\mathbb{E}(\w_j)\rightharpoonup\mathbb E(\w)$ weakly in $L^p(\Omega)$ for some $\w\in W^{1,p}(\Omega)$, thanks to Lemma \ref{corcurl}. By defining $\mathbf y_j:=\bb i+h_j\v_j$, we let $\mathbf R_j$ be the corresponding constant rotation matrix such that \eqref{muller} holds, so that by \eqref{coerc}, since $\mathcal F^I_{h_j}(\v_j)\le 1$, we get
\[
\int_\Omega g_p(|\nabla\mathbf y_j-\mathbf R_j|)\,dx\le \int_\Omega g_p(d(\nabla \mathbf y_j, SO(3)))\,dx\le
\int_\Omega \mathcal W(x,\mathbf I+h_j\nabla\v_j)\,dx\le h_j^2+h_j^2\mathcal L(\v_j).
\]
By inserting \eqref{elle} 
\[\begin{aligned}
\int_\Omega g_p(|\nabla\mathbf y_j-\mathbf R_j|)\,dx&\le h_j^2+h_j^2\mathcal L(\v_j)
\le h_j^2+C_\mathcal L \,h_j^2\,\|\mathbb E(\v_j)\|_{L^p(\Omega)}=h_j^2(1+C_{\mathcal L} \,t_j),
\end{aligned}\]
that is,
\begin{equation}\label{startingpoint}
\int_\Omega g_p(|\mathbf I-\mathbf R_j+t_jh_j\nabla \w_j|)\,dx\le h_j^2(1+C_{\mathcal L} \,t_j).
\end{equation}

We claim that $\nabla \w$ is the sum of a skew-symmetric matrix and an element of $\mathbb K$, where $\mathbb K$ is defined by \eqref{cone},
and that $\mathbb E(\w_j)\to \mathbb E(\w)$ in $L^p(\Omega)$, up to extraction of a further not relabeled subsequence. 
We shall prove the claim by  separately treating the following three possible cases: $h_jt_j\to\lambda\in(0,+\infty)$, $h_jt_j\to 0$ and $h_jt_j\to+\infty$ as $j\to+\infty$.

\textbf{Case 1: $h_j t_j\to \lambda$ as $j\to+\infty$ for some $\lambda>0$}.
It is easy to check that 
\begin{equation}\label{fromb}
2g_p(ax)\ge  (a^2\wedge a^p) g_p(x) \qquad\mbox{for any $x\ge 0 $ and any $a\ge 0$},\end{equation} 
so that \eqref{startingpoint} implies
\[\begin{aligned}
\frac12 ((h_j^2t_j^2)\wedge(h_j^pt_j^p))\int_\Omega g_p\left(\left|\frac{\mathbf I-\mathbf R_j}{h_jt_j}+\nabla \w_j\right|\right)\,dx&
\le \int_\Omega g_p\left(\left|{\mathbf I-\mathbf R_j}+h_jt_j\nabla \w_j\right|\right)\,dx\\&\le
 h_j^2(1+C_{\mathcal L}\,t_j),
\end{aligned}\]
therefore
\begin{equation}\label{limit1}
\lim_{j\to+\infty}\int_\Omega g_p\left(\left|\frac{\mathbf I-\mathbf R_j}{h_jt_j}+\nabla \w_j\right|\right)\,dx=0.
\end{equation}
We define $$A_j:=\{x\in\Omega: |\mathbf I-\mathbf R_j+t_jh_j\nabla \w_j(x)|\le t_jh_j \}.$$
Since $g_p(t)=t^2$ for $0\le t\le 1$ and $g_p(t)\ge t^p$ for $t\ge 1$, taking advantage of \eqref{limit1} we get
\[
\lim_{j\to+\infty}\int_\Omega \chi_{A_j} \left|\frac{\mathbf I-\mathbf R_j}{h_jt_j}+\nabla \w_j\right|^2\,dx+\lim_{j\to+\infty}\int_\Omega(1-\chi_{A_j})\,\left|\frac{\mathbf I-\mathbf R_j}{h_jt_j}+\nabla \w_j\right|^p\,dx=0
\]
so that both $$(1-\chi_{A_j})\, \left(\frac{\mathbf I-\mathbf R_j}{h_jt_j}+\nabla \w_j\right)\qquad\mbox{and}\qquad   \chi_{A_j} \left(\frac{\mathbf I-\mathbf R_j}{h_jt_j}+\nabla \w_j\right)$$ go to zero in $L^p(\Omega)$  as $j\to+\infty$, since $p\in(1,2]$. As a consequence, we obtain the convergence to zero of $h_j^{-1}t_j^{-1}{(\mathbf I-\mathbf R_j)}+\nabla \w_j$  in $L^p(\Omega)$ as $j\to+\infty$. Therefore, $\nabla \w_j$ converge in $L^p(\Omega)$, up to subsequences, to $\lambda^{-1}(\mathbf R-\mathbf I)$ for some suitable $\mathbf R\in SO(3)$ (thus $\mathbb E(\w_j)$ converge in $L^p(\Omega)$ to $\mathbb E(\w)$) and Lemma \ref{corcurl}  implies that $\nabla \w$ is the sum of a skew-symmetric matrix and an element of $\mathbb K$. 

\textbf{Case 2: $h_j t_j\to 0$ as $j\to+\infty$}. We assume wlog that $t_jh_j\le1$ for any $j\in\mathbb N$.
Writing $\mathbf R_j$ by means of the Euler-Rodrigues formula \eqref{eurod}, from \eqref{startingpoint} we get
$$
\int_\Omega g_p(|h_jt_j\nabla\w_j-\sin\theta_j \mathbf W_j-(1-\cos\theta)\mathbf W_j^2|)\,dx\le h_j^2(1+C_{\mathcal L} \,t_j),
$$ 
where, for any $j\in\mathbb N$, $\theta_j\in(-\pi,\pi]$ and $\mathbf W_j\in\mathbb R^{3\times3}$ is  skew-symmetric. Since $|\mathrm{sym}\mathbf F|\le |\mathbf F|$ and $g_p$ is increasing, we deduce
\[
\int_\Omega g_p(|h_jt_j\mathbb E(\w)-(1-\cos\theta)|\mathbf W_j^2|)\,dx\le h_j^2(1+C_{\mathcal L} \,t_j).
\]
Therefore, \eqref{fromb} implies (since $h_jt_j\le 1$)
\begin{equation}\label{fund}
\int_\Omega g_p\left(\left|\mathbb E(\w_j)-\frac{(1-\cos\theta_j)\mathbf W_j^2}{h_jt_j}\right|\right)\,dx\le \frac{2+2C_{\mathcal L} \,t_j}{t_j^2}.
\end{equation}
By taking advantage of the latter estimate,  since $g_p$ is increasing and satisfies $g_p(x)\le2 x^p$ for any $x\ge 0$, we get
\[
\begin{aligned}
|\Omega|\; g_p\left(\left|\frac{(1-\cos\theta_j)\mathbf W_j^2}{h_jt_j}\right|\right)&=\int_\Omega g_p\left(\left|\frac{(1-\cos\theta_j)\mathbf W_j^2}{h_jt_j}\right|\right)\,dx\\
&\le \int_\Omega g_p(|\mathbb E(\w_j)|)\,dx+\int_\Omega g_p\left(\left|\mathbb E(\w_j)-\frac{(1-\cos\theta_j)\mathbf W_j^2}{h_jt_j}\right|\right)\,dx\\
&\le 2\int_\Omega|\mathbb E(\w_j)|^p\,dx+ \int_\Omega g_p\left(\left|\mathbb E(\w_j)-\frac{(1-\cos\theta_j)\mathbf W_j^2}{h_jt_j}\right|\right)\,dx\\&\le 2+\frac{2+2C_{\mathcal L} \,t_j}{t_j^2}.
\end{aligned}
\] 
Since $t_j\to+\infty$ as $j\to+\infty$, we obtain the existence of a positive constant $C_*$ (not depending on $j$) such that
\[
\frac{1-\cos\theta_j}{h_jt_j}=\frac{\sqrt2}{2}\left|\frac{(1-\cos\theta_j)\mathbf W_j^2}{h_jt_j}\right|\le C_*.
\]
In particular, up to subsequences, we have
\[
\lim_{j\to+\infty} \frac{(1-\cos\theta_j)\mathbf W_j^2}{h_jt_j}=\mathbf G^2
\]
for some suitable constant skew-symmetric matrix $\mathbf G$, so that from \eqref{fund} we deduce, since $g_p$ is continuous and increasing,
\[\begin{aligned}
\lim_{j\to+\infty}\int_\Omega g_p(|\mathbb E(\w_j)-\mathbf G^2|)\,dx&\le\lim_{j\to+\infty}\int_\Omega g_p\left(\left|\mathbb E(\w_j)-\frac{(1-\cos\theta_j)\mathbf W_j^2}{h_jt_j}\right|\right)\,dx\\&\qquad+\lim_{j\to+\infty}\int_\Omega g_p\left(\left|\frac{(1-\cos\theta_j)\mathbf W_j^2}{h_jt_j}-\mathbf G^2\right|\right)\,dx=0.
\end{aligned}\]
By the same argument of Case 1, we conclude that $\mathbb E(\w_j)\to \mathbf G^2$ in $L^p(\Omega)$ as $j\to+\infty$, hence $\mathbb E(\w)=\mathbf G^2$. We deduce that the skew-symmetric part of $\nabla \w$ is a gradient field, hence a constant skew-symmetric matrix $\bb\Lambda$, and that $\nabla \w=\bb\Lambda+\tfrac12\mathbf G^2$. By applying the Euler-Rodrigues formula \eqref{eurod}, we deduce the existence of $\mu>0$, of $\mathbf R\in SO(3)$ and of $\mathbf Q\in \mathbb R^{3\times 3}_{\mathrm{skew}}$ such that $\nabla \w=\mu(\mathbf R-\mathbf I)+\mathbf Q$,
so that indeed $\nabla \w$ is the sum of an element of $\mathbb K$ and a skew-symmetric matrix.

\textbf{Case 3: $h_j t_j\to +\infty$ as $j\to+\infty$}. We may assume in this case that $h_jt_j\ge 1$ for any $j\in\mathbb N$. By applying \eqref{startingpoint} and \eqref{fromb}
we get
\[\begin{aligned}
\int_\Omega g_p\left(\left|\frac{\mathbf I-\mathbf R_j}{h_jt_j}+\nabla \w_j\right|\right)\,dx\le
 \frac{2h_j^2(1+C_{\mathcal L} \,t_j)}{h_j^pt_j^p},
\end{aligned}\]
where the right hand side vanishes as $j\to+\infty$, and where $\tfrac{\mathbf I-\mathbf R_j}{h_jt_j}$ vanishes as well, since $\mathbf R_j-\mathbf I$ is bounded. By the same argument of Case 1, we conclude that $\nabla \w_j\to\mathbf 0$ in $L^{p}(\Omega)$ as $j\to+\infty$,  thus $\mathbb E(\w_j)\to\mathbb E(\w)=\mathbf 0$ in $L^p(\Omega)$.
By Lemma \ref{corcurl}   we deduce that $\nabla\w$ is a constant skewsymmetric matrix. 
This ends the last of the three cases and proves the claim.\\

Let $\tilde\w_j:=\w_j-\mathbb P\w_j$ so that $\mathbb E(\tilde\w_j)=\mathbb E(\w_j)$ and $\mathcal L(\tilde \w_j)=\mathcal L(\w_j)$ by \eqref{globalequi}. The claim we just proved implies $\mathbb E(\tilde \w_j)\to\mathbb E(\w)$ in $L^p(\Omega)$ as $j\to+\infty$.
Therefore, since  \eqref{elle} implies 
\[\begin{aligned}
|\mathcal L(\tilde \w_j)-\mathcal L( \w)|&=|\mathcal L(\tilde\w_j- \w)|
\le C_{\mathcal L}\,
\|\mathbb E(\tilde \w_j)-\mathbb E(\w)\|_{L^p(\Omega)},
\end{aligned}\]
we deduce 
$\mathcal L( \w_j)=\mathcal L(\tilde\w_j)\to\mathcal L(\w) $ as $j\to+\infty$.
As a consequence, thanks to 
\eqref{quasi} and \eqref{ser} we infer that
\begin{equation}\label{limi}\begin{aligned}
\mathcal L(\w)&=\lim_{j\to+\infty}\mathcal L(\w_j)=\lim_{j\to+\infty}\frac1{t_j}\,\mathcal L(\v_j)\\&=\lim_{j\to+\infty}\frac1{t_j}\left(\frac{1}{h_j^2}\int_\om\mathcal W^I(x,\mathbf I+h_j\nabla\v_j)\,dx-\mathcal F_{h_j}^I(\v_j)\right)\ge 0. \end{aligned}
\end{equation}
Since we have already proven that $\nabla \w$ is the sum of a skewsymmetric matrix and an element of $\mathbb K$, we have
$\w(\x)=\tau(\Rot-\Id)\x+\mathbf A\x+\mathbf c$ for suitable $\tau\ge0$,  $\Rot\in SO(3)$ such that
 $\Rot\neq\Id$,   $\mathbf A\in\mathbf R^{3\times 3}_{\mathrm{skew}}$ and $\mathbf c\in \R^3.$ By
 \eqref{eurod} there exist   $\vartheta\in (-\pi,\pi]$, $\theta\neq 0$, and $\mathbf W\in \mathbb R^{3\times 3}_{\mathrm{skew}}$,  $\mathbf W\neq \mathbf 0$, such that  $\Rot=\Id+(1-\cos\vartheta)\mathbf W^{2}+(\sin\vartheta)\mathbf W$. Hence,
\eqref{globalequi} and \eqref{limi}
yield
\begin{equation*}\begin{aligned}
 0\le\mathcal L(\w)&=\tau\int_{\partial\Omega}\mathbf g\cdot(\Rot-\Id)x\,d\H^{N-1}+\tau\int_{\om}\mathbf f\cdot(\Rot-\Id)x\,dx
 \\&=\tau (1-\cos\vartheta)\left(\int_{\partial\Omega}\mathbf g\cdot\mathbf W^{2}x\,d\H^{N-1}+\int_{\om}\mathbf f\cdot\mathbf W^{2}x\,dx\right).
\end{aligned}
\end{equation*}
By taking \eqref{comp} into account, we conclude that $\tau=0$, so that $\nabla \w$ is a constant skew symmetric matrix
and then  $\E(\w)=\mathbf 0$. But  $\mathbb E(\w_j)\to \mathbb E(\w)$ in $L^p(\Omega)$ and $\|\mathbb E(\w_j)\|_{L^p(\Omega)}=1$ imply $\|\mathbb E(\w)\|_{L^{p}(\Omega)}=1$, a contradiction.
\end{proof}

\section{Lower bound}\label{Sectionpotential}

In this section we prove  the lower bound $\liminf_{j\to +\infty}\mathcal F_{h_j}^I(\v_j)\ge \mathcal F^I(\v)$ as $\mathbb E(\v_j)\wconv\mathbb E(\v)$ weakly in $L^p(\om,\R^{3\times3})$ and $h_j\to0$. 
We start with two preliminary lemmas. 

\begin{lemma}\lab{sqrth} Assume \eqref{OMEGA},\eqref{globalequi}, \eqref{infty}, \eqref{framind}, \eqref{Z1}, \eqref{reg},    
\eqref{coerc}. Let $(h_j)_{j\in\mathbb N}\subset(0,1)$ be a vanishing sequence   and let $(\v_j)_{j\in\mathbb N}\subset W^{1,p}(\Omega,\mathbb R^3)$ be a sequence such that $\sup_{j\in\mathbb N}{\mathcal F^I_{h_j}}(\v_j)<+\infty$ and such that  $\mathbb E(\v_j)\wconv \mathbb E(\v)$ weakly in $L^p(\Omega,\mathbb R^{3\times3})$ as $j\to\infty$ for some $\v\in W^{1,p}(\Omega,\mathbb R^3)$.
Then there exists a constant matrix $\mathbf W\in \R^{3\times 3}_{\rm skew}$ such that, up to subsequences, $\sqrt h_j\nabla\v_j\to \mathbf W$ in $L^p(\Omega,\mathbb R^{3\times3})$ as $j\to\infty$. 
\end{lemma}
\begin{proof}
Through the proof, $C$ will always denote a generic positive constant only depending on $p$, $\Omega$, $\mathbf f$ and $\mathbf g$. 
By taking into account \eqref{globalequi}, \eqref{coerc}, and by applying \eqref{muller} to $\yy_j:=\x+h_j \v_j$, we see that for any $j\in\mathbb N$ there exists a constant matrix $\mathbf R_j\in SO(3)$ such that
$$h_j^{-2}\int_\Omega g_p(\vert \mathbf I+h_j\nabla\v_j-\mathbf R_j\vert)\,dx-
\mathcal L(\v_j-\mathbb P\v_j) \ \le \
C.
$$
By \eqref{elle} and by the boundedness of $\mathbb E(\v_j)$ in $L^p(\Omega,\mathbb R^{3\times3})$ we get
\begin{equation}
\label{bound}
h_j^{-2}\int_\Omega g_p(\vert \mathbf I+h_j\nabla\v_j-\mathbf R_j\vert)\,dx \leq\ C
\end{equation}

Due to the representation \eqref{eurod} of rotations,
for every $j\in \mathbb N$ there exist $\vartheta_j\in (-\pi,\pi]$ 
and
$\mathbf W_j\!\in \!\mathbb R^{3\times 3}_{\rm skew}$, with 
$\vert\mathbf W_{j}\vert^{2}=2$, such that
\begin{equation*}\lab{euform}
 \mathbf R_j\,=\,\exp(\vartheta_j \mathbf W_j)\,=\,
\mathbf I\,+\,\sin\vartheta_j\,\mathbf W_j\, +\,(1-\cos\vartheta_j)\,\mathbf W_j^{2}.
\end{equation*}
Hence, by \eqref{bound} and \eqref{gp}, we have
\begin{equation}
\label{bound2}
\int_\Omega g_p(\vert h_j\nabla\v_j-\sin\vartheta_j \mathbf W_j-(1-\cos\vartheta_j)\mathbf W_j^2\vert)\,dx\leq Ch_{j}^2\,.
\end{equation}
By setting 
\begin{equation*}
A_j:=\{x\in\om: \vert h_j\nabla\v_j-\sin\vartheta_j \mathbf W_j-(1-\cos\vartheta_j)\mathbf W_j^2\vert\le 1\},
\end{equation*}
from \eqref{bound2} we get via H\"older inequality
\begin{equation}\label{ext1}\begin{aligned}
& \int_{A_j} \vert h_j\nabla\v_j-\sin\vartheta_j \mathbf W_j-(1-\cos\vartheta_j)\mathbf W_j^2\vert^p\,dx
\\&\qquad
\le   \left\{\int_{A_j} \vert h_j\nabla\v_j-\sin\vartheta_j \mathbf W_j-(1-\cos\vartheta_j)\mathbf W_j^2\vert^2\,dx\right\}^{p/2}|A_j|^{1-\frac{p}{2}}\le Ch_j^p,
\end{aligned}\end{equation}
and again by \eqref{gp} we also obtain
\begin{equation}\label{ext2}
\int_{\Omega\setminus A_j} \vert h_j\nabla\v_j-\sin\vartheta_j \mathbf W_j-(1-\cos\vartheta_j)\mathbf W_j^2\vert^p\,dx\leq Ch_{j}^2\le Ch_j^p
\end{equation}
Since
$${\rm sym} \Big(h_j\nabla \v_j-\sin\vartheta_j \mathbf W_j
-(1-\cos\vartheta_j)\mathbf W_j^2\Big)=h_{j}\mathbb E(\v_j)-(1-\cos\vartheta_j)\mathbf W_j^2$$
 we get
$$\int_\Omega\vert \mathbb E(\v_j)-(1-\cos\vartheta_j)h_{j}^{-1}\mathbf W_j^2\vert^{p}\,dx
\leq
{C}
.$$
By recalling that $\mathbb E(\v_j)\wconv\mathbb E(\v)$ in $L^p(\Omega,\mathbb R^{3\times 3})$ we get
\begin{equation}
\label{bound3}
\left | 1-\cos\vartheta_j\right |\leq C {h_{j}}
\end{equation}
hence
\begin{equation}
\label{bound4}
\vert\sin \vartheta_j\vert\leq \sqrt{2Ch_{j}}.
\end{equation}
%
%
By \eqref{ext1} and \eqref{ext2} we have
\begin{equation*}
\int_{\Omega} \vert \sqrt{h_j}\nabla\v_j-h_j^{-1/2}\sin\vartheta_j \mathbf W_j-h_j^{-1/2}(1-\cos\vartheta_j)\mathbf W_j^2\vert^p\,dx\leq Ch_{j}^{p/2},
\end{equation*}
therefore, 
by \eqref{bound3} and \eqref{bound4}
there exists a constant matrix
$\mathbf W\in \R^{3\times 3}_{\rm skew}$ such that, up to subsequences,
  $\sqrt{h_{j}}\nabla\v_j\ \rightarrow\ \mathbf W$ in $L^p(\Omega,\mathbb R^{3 \times 3})$ as $j\to\infty$.
\end{proof}

\begin{lemma}\lab{Bj}Assume \eqref{OMEGA}, \eqref{globalequi}, \eqref{infty}, \eqref{framind}, \eqref{Z1}, \eqref{reg},    
\eqref{coerc}. Let $(h_j)_{j\in\mathbb N}\subset(0,1)$ be a vanishing sequence and let $(\v_j)_{j\in\mathbb N}\subset W^{1,p}(\Omega, \mathbb R^3)$ be a sequence  such that $\sup_{j\in\mathbb N}{\mathcal F^I_{h_j}}(\v_j)<+\infty$ and $\mathbb E(\v_j)\wconv \mathbb E(\v)$ weakly in $L^p(\Omega,\mathbb R^{3\times3})$ as $j\to\infty$ for some $\v\in W^{1,p}(\Omega,\mathbb R^3)$. Then $\v\in H^1(\Omega,\mathbb R^3)$ and up to subsequnces ${\mathbf 1}_{B_j}\mathbb E(\v_j)\wconv \mathbb E(\v)$ weakly in $L^2(\Omega,\mathbb R^{3\times3})$ as $j\to\infty$, where
\begin{equation}\label{lab} B_j:=\left\{x\in \om: |\sqrt h_j\nabla\v_j|\le 2|\mathbf W|+1\right\}\end{equation}
	 and where $\mathbf W\in\mathbb R^{3\times3}_{\mathrm{skew}}$ is given by   { \rm Lemma \ref{sqrth}}.  \KKK
\end{lemma}
\begin{proof} Again, through the proof we denote by $C$ the various  positive constants, possibly  depending only on $\Omega, p, \mathbf f, \mathbf g$.
Let $\theta_j$ and $\mathbf W_j$ be as in the proof of Lemma \ref{sqrth}, so that \eqref{bound3} holds. \KKK
We claim that ${\mathbf 1}_{B_j}\mathbb E(\v_j)$ is bounded in $L^2(\Omega,\mathbb R^{3\times 3})$. To this aim it is useful to notice that by \eqref{gp} for every $\delta>0$ there exists $c=c(p,\delta)>0$ such that $g_p(t)\ge ct^2$ for every $t\in [0,\delta]$. Therefore by taking into account that for $j$ large enough 
\begin{equation*}
{\mathbf 1}_{B_j}|h_j\mathbb E(\v_j)-(1-\cos\theta_j){\mathbf W}_j^2|\le 2|\mathbf W|+1.
\end{equation*}
we may fix $\delta=2|\mathbf W|+1$ and obtain for any $j$ large enough 
\begin{equation*}\begin{aligned}
&\displaystyle\int_{B_j}|\mathbb E(\v_j)|^2\,dx\le 2h_j^{-2}\int_{B_j}|h_j\mathbb E(\v_j)-(1-\cos\theta_j)\mathbf W_j^2|^2\,dx+2|B_j|\,h_j^{-2}\,|(1-\cos\theta_j)\mathbf W_j^2|^2\\
&\displaystyle\qquad \le Ch_j^{-2}\int_{B_j}g_p(|h_j\mathbb E(\v_j)-(1-\cos\theta_j){\mathbf W}_j^2|)\,dx+2|B_j|\,h_j^{-2}\,|(1-\cos\theta_j){\mathbf W}_j^2|^2\\
&\displaystyle\qquad \le Ch_j^{-2}\int_{B_j}g_p(|h_j \nabla\v_j-(1-\cos\theta_j)\mathbf W_j^2-\sin\theta_j{\mathbf W}_j|)\,dx+2{h_j}^{-2}|B_j||(1-\cos\theta_j)\mathbf W_j^2|^2\\
\end{aligned}
\end{equation*}
and by arguing as in the proof of Lemma \ref{sqrth}, see \eqref{bound2} and \eqref{bound3}, we get 
\[\displaystyle\int_{B_j}|\mathbb E(\v_j)|^2\,dx\le C\]
as claimed. On the other hand for every $q\in (1,p)$ we have
\begin{equation*}
\displaystyle\int_{B_j^c}|\mathbb E(\v_j)|^q\,dx\le \left(\int_{B_j^c}|\mathbb E(\v_j)|^p\,dx\right)^{q/p}|B_j^c|^{(p-q)/p}\to 0
\end{equation*}
since $|B_j^c|\to 0$ by Chebyshev inequality and by Lemma \ref{sqrth}. 
By taking into account that
\begin{equation*}
\mathbb E(\v_j)={\mathbf 1}_{B_j^c}\mathbb E(\v_j)+{\mathbf 1}_{B_j}\mathbb E(\v_j)
\end{equation*}
and by assuming wlog that ${\mathbf 1}_{B_j}\mathbb E(\v_j)\wconv \u$ weakly in $L^2(\om,\R^{3\times3})$ we get $\mathbb E(\v_j)\wconv \u$
weakly in $L^q(\om,\R^{3\times3})$ and recalling that $\mathbb E(\v_j)\wconv \mathbb E(\v)$ weakly in $L^p(\om,\mathbb R^{3\times3})$ we get $\u=\mathbb E(\v)\in L^2(\om,\R^{3\times3})$ thus proving that $\v\in H^1(\om,\R^3)$  by Korn inequality. \KKK
\end{proof}
 \begin{lemma}\label{lowerbd}{\bf (Lower bound)}. Assume \eqref{OMEGA}, \eqref{infty}, \eqref{framind}, \eqref{Z1}, \eqref{reg},    
\eqref{coerc}. Let $(h_j)_{j\in\mathbb N}\subset(0,1)$ be a vanishing sequence  and let $(\v_j)_{j\in\mathbb N}\subset W^{1,p}(\Omega,\mathbb R^3)$ be a sequence such that $ \mathbb E(\v_j)\wconv \mathbb E(\v)$ weakly in $L^p(\Omega,\mathbb R^{3\times3})$ as $j\to\infty$. Then
\begin{equation}\lab{liminf}
\displaystyle\liminf_{j\to +\infty}\mathcal F_{h_j}^I(\v_j)\ge \mathcal F^I(\v).
\end{equation}
\end{lemma}
\begin{proof} We may assume wlog that ${\mathcal F^I_{h_j}}(\v_j)\le C$ for any $j\in\mathbb N$ so that by  setting $$\textstyle\mathbf D_{j}:=\mathbb E(\v_{j})+\frac{1}{2}h_{j}\nabla\v_{j}^{T}\nabla\v_{j}$$ we get  
\begin{equation*}\begin{aligned}
1&=\det(\mathbf I+h_{j}\nabla\v_{j})=\det(\mathbf I+h_{j}\nabla\v_{j}^T)(\mathbf I+h_{j}\nabla\v_{j})=\det(\mathbf I+2h_{j}\mathbb E(\v_{j})+h_{j}^{2}\nabla\v_{j}^{T}\nabla\v_{j})\\
&= 1+2h_{j}\mathrm{Tr} \mathbf D_{j}-{2}h_{j}^{2}(\mathrm{Tr}(\mathbf D_{j}^{2})-(Tr \mathbf D_{j})^{2})+8h_{j}^{3}\det \mathbf D_{j}
\end{aligned}
\end{equation*}
a.e. in $\om$, that is, 
\begin{equation*}\lab{TrB}
2\dv \v_{j}+h_{j}|\nabla\v_{j}|^{2}=
2Tr \mathbf D_{j}={2}h_{j}(Tr(\mathbf D_{j}^{2})-(Tr \mathbf D_{j})^{2})-8h_{j}^{2}\det \mathbf D_{j}.
\end{equation*}
By taking into account Lemma \eqref{sqrth} we get $\sqrt h_j\nabla\v_j\to \mathbf W$ in $L^p$ hence, up to subsequences, $h_{j}\nabla\v_{j}^T\nabla\v_j\to -\mathbf W^2$ a.e. in $\om$ and ${2}h_{j}(Tr(\mathbf D_{j}^{2})-(Tr \mathbf D_{j})^{2})-8h_{j}^{2}\det \mathbf D_{j}\to 0$ a.e. in $\om$. \KKK
Therefore $2\dv\v_j\to \Tr\mathbf W^2$ a.e. in $\om$ and since the weak convergence of $\mathbb E(\v_j)$ implies $\dv \v_{j}\wconv \dv \v$ weakly in $L^p(\Omega)$ we get $2\dv\v=\Tr\mathbf W^2$ a.e. in $\om$.
On the other hand by Lemma \ref{Bj}, with $B_j$ defined by \eqref{lab}, we have   ${\mathbf 1}_{B_j}\mathbb E(\v_j)\wconv \mathbb E(\v)$ weakly in $L^2(\Omega,\mathbb R^{3\times3})$ and $\v\in H^1(\Omega,\mathbb R^3)$.  Hence, by \eqref{regV}, \eqref{vi} and \eqref{globalequi} we get for large enough $j$
\begin{equation*}\begin{aligned}
\mathcal F_{h_j}^I(\v_j)&\ge \frac1{h_j^2}\int_{B_j}\mathcal W(x, \mathbf I+h_j\nabla\v_j)\,dx-\mathcal L(\v_j-\mathbb P\v_j)=\frac1{h_j^2}\int_{B_j}\mathcal V(x,h_j \mathbf D_j)\,dx-\mathcal L(\v_j-\mathbb P\v_j)\\
&\displaystyle\ge \int_{B_j}\frac12 \mathbf D_j^T\,D^2\mathcal V(x,\mathbf 0)\,\mathbf D_j\,dx-\int_{B_j} \eta(h_j\mathbf D_j)|\mathbf D_j|^2\,dx-\mathcal L(\v_j-\mathbb P\v_j),
\end{aligned}
\end{equation*}
since on $B_j$ we have $h_j|\mathbf D_j|\le \sqrt{h_j}\left(\sqrt{h_j}|\nabla\v_j|+\tfrac12h_j^{3/2}\nabla \v_j^T\nabla \v_j\right)\le \sqrt{h_j}(2|\mathbf W|+2)$ for large enough $j$. We deduce
\begin{equation}\begin{aligned}\lab{lowerbd}
\mathcal F_{h_j}^I(\v_j)
\displaystyle&\ge \frac12\int_\om ({\mathbf 1}_{B_j}\mathbf D_j)^T\,D^2\mathcal W(x,\mathbf I)\,(\mathbf 1_{B_j}\mathbf D_j)\,dx\\&\qquad- \eta(\sqrt h_j(2|\mathbf W|+1)\int_\om|{\mathbf 1}_{B_j}\mathbf D_j|^2\,dx-\mathcal L(\v_j-\mathbb P\v_j)
\end{aligned}
\end{equation}
for large enough $j$, as $\eta$ is increasing. \KKK
Since $h_{j}\nabla\v_{j}^T\nabla\v_j\to -\mathbf W^2$ a.e. in $\om$ and $|B_j^c|\to 0$ as $j\to+\infty$, and since $|{\mathbf 1}_{B_j}h_{j}\nabla\v_{j}^T\nabla\v_j| \le (2|\mathbf W|+1)^2\KKK$, we get ${\mathbf 1}_{B_j}h_{j}\nabla\v_{j}^T\nabla\v_j\wconv -\mathbf W^2$ weakly in $L^2(\Omega,\mathbf R^{3\times3})$. By  taking into account that ${\mathbf 1}_{B_j}\mathbb E(\v_j)\wconv \mathbb E(\v)$ weakly in $L^2(\Omega,\mathbf R^{3\times3})$, we then obtain
$\textstyle{\mathbf 1}_{B_j}\mathbf D_j\wconv \mathbb E(\v)-\frac{1}{2}\mathbf W^2$ weakly in $L^2(\Omega,\mathbb R^{3\times 3})$. 
Hence, by \eqref{reg}, \eqref{lowerbd} and by the weak $L^2(\Omega,\mathbb R^{3\times3})$ lower semicontinuity of the map $\mathbf F\mapsto\int_\Omega \mathbf F^T\, D^2\mathcal W(x,\mathbf I)\,\mathbf F\,dx$, we deduce
\begin{equation}\label{pi}
\displaystyle\liminf_{j\to +\infty}\mathcal F_{h_j}^I(\v_j)\ge \frac12\int_\om \left(\mathbb E(\v)-\frac{1}{2}\mathbf W^2\right)\,D^2\mathcal W(x,\mathbf I)\,\left(\mathbb E(\v)-\frac{1}{2}\mathbf W^2\right)\,dx-\mathcal L(\v).
\end{equation}
In order to obtain \eqref{pi}, we have also used the fact that $\mathcal L(\v_j-\mathbb P\v_j)\to\mathcal L(\v)$ as $j\to+\infty$. Indeed, by \eqref{secondkorn} and \eqref{kornpoi}  we obtain boundedness of $\v_j-\mathbb P\v_j$ in $W^{1,p}(\Omega,\mathbb R^3)$ and in $L^{\frac{3p}{3-p}}(\Omega,\mathbb R^3)$. Therefore up to subsequences we get $\v_j-\mathbb P\v_j\wconv\mathbf w$ weakly in $L^{\frac{3p}{3-p}}(\Omega,\mathbb R^3)$ for some $\mathbf w\in W^{1,p}(\Omega,\mathbb R^3)$, and by trace embedding $\v_j-\mathbb P\v_j\wconv \w$ weakly in $L^{\frac{2p}{3-p}}(\partial\Omega,\mathbb R^3)$ as well. Moreover, since  $\mathbb E(\v_j-\mathbb P\v_j)\wconv\mathbb E(\v)$ weakly in $L^{p}(\Omega,\mathbb R^{3\times3})$, we deduce $\mathbb E(\mathbf w)=\mathbb E(\v)$, thus \eqref{globalequi} implies $\mathcal L(\mathbf w)=\mathcal L(\v)$. Therefore, we have $\mathcal L(\v_j-\mathbb P\v_j)\to\mathcal L(\mathbf w)=\mathcal L(\v)$ as $j\to+\infty$. 
   
Eventually, since $2\dv\v=\Tr\mathbf W^2$ a.e. in $\om$, the result follows from \eqref{pi}.
\end{proof}

\section{Upper bound}

The following result is an extension of \cite[Lemma 4.1]{MP}.
\begin{lemma}\label{newflowlemma}
Let $\Omega$ satisfy assumption \eqref{OMEGA}.   Let $\v\in W^{2,\infty}(\Omega',\R^3)$ be such that $\mathrm{div}\, \v=0$  in $\om'$, where $\Omega'\subset\mathbb R^3$ is an open set such that $\overline\Omega\subset\Omega'$.
Let  $(h_j)_{j\in\mathbb N}\subset(0,1)$ be a vanishing sequence. There exists a sequence
 of vector fields $(\v_j)_{j\in\mathbb N}\subset W^{2,\infty}(\Omega,\mathbb R^3)$ and $j_*\in\mathbb N$ such that for any $j>j_*$
\begin{equation}
\det(\mathbf I+h_j\nabla \v_j)=1,\label{flux1}
\end{equation}
\begin{equation}
 \sup_{x\in\Omega}|\v_{j}(x)- \v(x) |\le \|\v\|_{L^\infty(\Omega)}\,\mathfrak q(h_j\|\v\|_{W^{1,\infty}(\Omega)}),\label{flux2}
\end{equation}
\begin{equation}
\sup_{x\in\Omega}|h_j\nabla\v_j(x)|\le \mathfrak q(h_j\|\v\|_{W^{1,\infty}(\Omega)}),\label{flux3}
\end{equation}
\begin{equation}
\sup_{x\in\Omega}|\nabla\v_j(x)-\nabla \v(x)|
\le \left(1+e^{h_j\|\v\|_{W^{1,\infty}(\Omega')}}\right)\,\|\v\|_{W^{2,\infty}(\Omega')}\,\mathfrak q(h_j\|\v\|_{W^{1,\infty}(\Omega')}),\label{flux4}
\end{equation}
where $\mathfrak q(z):=ze^z$. In particular, $\v_j\to \v$ in $W^{1,\infty}(\Omega,\mathbb R^3)$ as $j\to+\infty$.
\end{lemma}
\begin{proof}
We choose $T\in(0,1)$ small enough, such that $\mathbf y(t,x)\in\om'$ for any $x\in\om$ and any $t\in[0,T]$, where $\mathbf y(\cdot,x)$ is the unique solution to
\beeq \lab{flowbis}\left\{\begin{array}{ll} &\displaystyle\frac{\partial  {\bf y}}{\partial t}(t,x)=\v({\bf y}(t,x)),\qquad t\in(0,T]\\
&\\
& {\bf y}(0,x)=x,
\end{array}\right.
\eneq
so that $\mathbf y$ is  
 the flow associated to the vector field $\v$. 
We have
$\mathbf y\in C^1([0,T]; W^{2,\infty}(\om))$, see \cite[Corollary 5.2.8, Remark 5.2.9]{HP}.
From \eqref{flowbis}, we have 
\begin{equation}\label{diffbis}
\frac1t\left(\mathbf y(t,x)-x\right)-\mathbf v(x)=\frac1t\int_0^t(\mathbf v(\mathbf y(s,x))-\mathbf v(x))\,ds
\end{equation}
for any $x\in\Omega$. We get therefore the basic estimate
\[
\frac1t|\mathbf y(t,x)-x|\le |\v(x)|+ \|\mathbf v\|_{W^{1,\infty}(\Omega')}\int_0^t\frac1s|\mathbf y(s,x)-x|\,ds
\]
for any $x\in\Omega$,
and Gronwall lemma entails 
\begin{equation}\label{gronw}
\frac1t|\mathbf y(t,x)-x|\le |\mathbf v(x)| \exp\{\|\mathbf v\|_{W^{1,\infty}(\Omega')}t\},
\end{equation}
so that we have 
\begin{equation}\label{supomega}
\sup_{x\in\Omega}|\mathbf y(t,x)-x|\le \mathfrak q (t\|\mathbf v\|_{W^{1,\infty}(\Omega')}).
\end{equation}

We have $\nabla \mathbf y\in C^1([0,T]; W^{1,\infty}(\om))$, where $\nabla$ denotes the gradient in  $x$, and as shown in the proof of \cite[Lemma 4.1]{MP}, there hold
\begin{equation}\label{QZ}
\nabla\mathbf y(t,x)=\exp\left (\int_0^t\nabla\v(\mathbf y(s,x))\,ds\right )
\end{equation}
 and then
 \begin{equation}\lab{detbis}
\det\nabla\mathbf y (t,x)=\exp\left (\int_0^t\tr\nabla\v(\mathbf y(s,x))\,ds\right )=\exp\left (\int_0^t\dv\v(\mathbf y(s,x))\,ds\right )=1
\end{equation}
for every $t\in (0,T]$ and for every $x\in\om$. 



We define $$  \v_{t}(x):= t^{-1}({\bf y}(t,x)-x),\qquad x\in\om,\; t\in(0,T].$$
From the definition of $\mathbf v_t$, from \eqref{diffbis} and \eqref{gronw} we get
\begin{equation}\label{ellinf}\begin{aligned}
\displaystyle\left |\v_{t}(x)- \v(x)\right |&=\left|\frac1{t}(\mathbf y(t,x)-x)-\mathbf v(x)\right|	\le  \|  \v\|_{W^{1,\infty}(\Omega)}\, \int_{0}^{t}\frac1s\,\left |{\bf y}(s,x)-x\right |\,ds\\&\le \|\v\|_{L^\infty(\Omega)}\mathfrak q(t\|\v\|_{W^{1,\infty}(\Omega)})
\end{aligned}\end{equation}
for any $x\in\Omega$ and any $t\in(0,T]$. From the latter we get in particular the convergence of $\v_t$ to $\v$ in $L^1\cap L^\infty(\Omega)$ as $t\to0$.

 Since the map $\Omega\ni x\mapsto\v(\mathbf y(t,x))$ is Lipschitz continuous, uniformly with respect to  $t\in(0,T)$, 
we may take the gradient under integral sign in \eqref{diffbis} and obtain
\begin{equation}\label{gradbis}
\begin{aligned}
\frac1t (\nabla \mathbf y(t,x)-\mathbf I)-\nabla\mathbf v(x)&=\frac1t\int_0^t\left(\nabla[\mathbf v(\mathbf y(s,x))]-\nabla \mathbf v(x)\right)\,ds\\
&=\frac1t \int_0^t\left(\nabla\mathbf v(\mathbf y(s,x))\nabla \mathbf y(s,x)-\nabla \mathbf v(x)\nabla\mathbf y(s,x)\right)\,ds\\&\qquad\qquad+
\frac1t\int_0^t\left(\nabla \mathbf v(x)\nabla\mathbf y(s,x) -\nabla \mathbf v(x)\right)\,ds
\end{aligned}
\end{equation}
for every $x\in\Omega$ and every $t\in(0,T]$.  Form the first equality of \eqref{gradbis} and from \eqref{QZ} we get  
\begin{equation*}\label{ad4}\begin{aligned}
\frac1t |\nabla \mathbf y(t,x)-\mathbf I|&\le\frac 1t \int_0^t|\nabla \v(\mathbf y(s,x))|\,|\nabla\mathbf y(s,x)|\,ds\\&\le \|\v\|_{W^{1,\infty}(\om')} \,\frac1t\,\int_0^t |\nabla\mathbf y(s,x)|\,ds\le \|\v\|_{W^{1,\infty}(\om')} \frac1t\,\int_0^t\exp\{s\|\v\|_{W^{1,\infty}(\Omega')}\}\,ds\\
&\le \|\v\|_{W^{1,\infty}(\Omega')}\,\exp\{t\|\v\|_{W^{1,\infty}(\Omega')}\},
\end{aligned}\end{equation*}
%
%
therefore
\begin{equation}\label{lastflow}
\sup_{x\in\Omega}|\nabla \mathbf y(t,x)-\mathbf I|\le \mathfrak q(t\|\v\|_{W^{1,\infty}(\Omega')})
\end{equation}
for any $t\in(0,T]$. Moreover, by \eqref{gradbis}, \eqref{supomega}, \eqref{lastflow} and \eqref{QZ} we have
\begin{equation*}
\begin{aligned}
&|\nabla\v_t(x)-\nabla \v(x)|=\left|\frac1{t}(\nabla\mathbf y(t,x)-\mathbf I)-\nabla \v(x)\right|
\\&\le \frac1t\int_0^{t}{|\nabla \mathbf y(s,x)|}\,|\nabla \v(\mathbf y(s,x))-\nabla \v(x)|\,ds+\| \v\|_{W^{1,\infty}(\om')}\,\frac1t\int_0^{t}{|\nabla  \mathbf y(s,x)-\mathbf I|}\,ds\\
&\le \|\v\|_{W^{2,\infty}(\Omega')}\,\frac1t\,\int_0^t |\nabla \mathbf y(s,x)|\,|\mathbf y(s,x)-x|\,ds+\|\v\|_{W^{1,\infty}(\Omega')}\,\mathfrak q(t\|\v\|_{W^{1,\infty}(\Omega')})\\
&  \le \|\v\|_{W^{2,\infty}(\Omega')}\,\frac1t\int_0^t\exp\{s \|\v\|_{W^{1,\infty}(\Omega')}\}\,\mathfrak q(s \|\v\|_{W^{1,\infty}(\Omega')})\,ds +\|\v\|_{W^{1,\infty}(\Omega')}\,\mathfrak q(t\|\v\|_{W^{1,\infty}(\Omega')})\\
&\le \left(e^{t\|\v\|_{W^{1,\infty}(\Omega')}}\right)\,\|\v\|_{W^{2,\infty}(\Omega')}\,\mathfrak q(t\|\v\|_{W^{1,\infty}(\Omega')})
\end{aligned}
\end{equation*}
for any $x\in\Omega$ and any $t\in(0,T]$. 

Eventually, let us consider a vanishing sequence $(h_j)_{j\in\mathbb N}\subset(0,1)$. By defining $j_*$ as the smallest positive integer such that $h_j<T$ for any $j>j_*$  and by defining $\v_j:=\v_{h_j}$, the result follows from \eqref{detbis}, \eqref{ellinf}, \eqref{lastflow} and from the latter estimate. 
\end{proof}

We next provide the approximation construction for the recovery sequence.

\begin{lemma}\label{Kato} Let $\om$ satisfy assumption \eqref{OMEGA}.
Let $(h_j)_{j\in\mathbb N}$ be a vanishing sequence.
Let $\v\in H^1_{\dv}(\Omega)$. There exists a sequence $(\v_j)_{j\in\mathbb N}\subset W^{2,\infty}(\om,\mathbb R^3)$ such that
\begin{itemize}
\item[i)]
 $\det(\mathbf I+h_j\nabla \v_j)=1$ for any $j\in\mathbb N$,
 \item[ii)]  $h_j\|\nabla\v_j\|_{L^\infty(\Omega)}\to 0$ as $j\to+\infty$,
  \item[iii)]  $\v_j\to\v$ strongly in $H^1(\Omega,\mathbb R^3)$ as $j\to+\infty$.
 \end{itemize}
\end{lemma}
\begin{proof}
Assumption \eqref{OMEGA} implies that there are $m$ bounded open connected Lipschitz sets with connected boundary, denoted by $\Omega_*,\Omega_1,\ldots,\Omega_{m-1}$, such that $\Omega=\Omega_*\setminus(\overline\Omega_1\cup\ldots\cup\ldots \overline \Omega_{m-1})$.  We have $\Omega_*\equiv\Omega$ if $m=1$.
By  \cite[Corollary 3.2]{Kato}, we may extend $\v$ to a $H^1(\mathbb R^3\setminus \overline{P_\delta},\mathbb R^3)$ vector field, still denoted by $\v$, such that $\dv\v=0$ on $\mathbb R^3 \setminus\overline{P_\delta}$, where $P:=\{p_1,\ldots, p_{m-1}\}$, $p_i\in\Omega_i$ for any $i=1,\ldots, m-1$, and $P_\delta:=B_\delta(p_1)\cup\ldots\cup B_\delta(p_{m-1})$. We have $P=P_\delta
=\emptyset$ if $m=1$,  otherwise $\delta>0$ is  fixed and so small that $\overline{B_\delta(p_i)}\subset \Omega_i$ for any $i=1,\ldots, m-1$. We note that this extension can be constructed in the form
$
\v=\dv \mathbf V+\lambda_1\mathbf g_1+\ldots+\lambda_{m-1}\mathbf g_{m-1}
$,
where $\mathbf V\in H^2(\mathbb R^3,\mathbb R^{3\times 3})$,  $\lambda_i\in \mathbb R$ and $\mathbf g_i$ is a  harmonic field with singularity at $p_i$ for each $i=1,\ldots,m-1$.
 
 Let $\rho$ denote the standard unit symmetric mollifier on $\mathbb R^3$ and let $\rho_\eps(x):=\eps^{-3}\rho(x/\eps)$. We let $\v^\eps:=\v\ast\rho_\eps$ and we let $\Omega'$ be a bounded open set such that $\overline\Omega\subset\Omega'$. We may choose $\Omega'$ such that $\v^\eps\in C^{\infty}(\overline{\Omega'})$, such that $\overline\Omega\subset\Omega'\subset\subset\mathbb R^3\setminus \overline{P_\delta}$ and such that $\dv\v^\eps\equiv0$ in $\Omega'$ for any small enough $\eps$. 
 Therefore, we may fix a constant $c\in(0,1)$, only depending on $\delta$, such that $\eps_n:=c/n$ satisfies the above properties for any $n\in\mathbb N$.

 Given 
 a vanishing sequence $(h_j)_{j\in\mathbb N}$ of positive numbers and given $n\in\mathbb N$ we may define $(\v^{\eps_n}_j)_{j\in\mathbb N}$ to be the sequence from  Lemma \ref{newflowlemma}, constructed from the $W^{2,\infty}(\Omega',\mathbb R^3)$ divergence-free vector field $\v^{\eps_n}$: indeed, by applying Lemma \ref{newflowlemma} to $\v^{\eps_n}$, we construct   a sequence  of $W^{2,\infty}(\Omega)$ vector fields $(\v^{\eps_n}_j)_{j\in\mathbb N}$ and a strictly increasing diverging function $g:\mathbb N\to\mathbb N$ such that   \eqref{flux1}, \eqref{flux2}, \eqref{flux3}, \eqref{flux4} 
 are satisfied for any positive integers $j,n$ with $j>g(n)$, i.e.,
 \begin{equation}
\det(\mathbf I+h_j\nabla \v^{\eps_n}_j)=1,\label{flux11}
\end{equation}
\begin{equation}
 \sup_{x\in\Omega}|\v_{j}^{\eps_n}(x)- \v^{\eps_n}(x) |\le \|\v^{\eps_n}\|_{L^\infty(\Omega')}\,\mathfrak q(h_j\|\v^{\eps_n}\|_{W^{1,\infty}(\Omega')}),\label{flux22}
\end{equation}
\begin{equation}
\sup_{x\in\Omega}|h_j\nabla\v^{\eps_n}_j(x)|\le \mathfrak q(h_j\|\v^{\eps_n}\|_{W^{1,\infty}(\Omega')}),\label{flux33}
\end{equation}
\begin{equation}
\sup_{x\in\Omega}|\nabla\v^{\eps_n}_j(x)-\nabla \v^{\eps_n}(x)|
\le \left(1+e^{h_j\|\v^{\eps_n}\|_{W^{1,\infty}(\Omega')}}\right)\,\|\v^{\eps_n}\|_{W^{2,\infty}(\Omega')}\,\mathfrak q(h_j\|\v^{\eps_n}\|_{W^{1,\infty}(\Omega')}).\label{flux44}
\end{equation}
 Therefore,  by defining the diverging sequence $(n(j))_{j\in\mathbb N}\subset [0,+\infty)$ as $$n(j):=\min\{g^{-1}(j), h_j^{-1/10}\}-1,$$  we obtain that   \eqref{flux11}, \eqref{flux22}, \eqref{flux33}, \eqref{flux44} are satisfied, with $\eps_{n(j)}$ in place of $\eps_n$,  for any $j\in\mathbb N$, since $n(j)<g^{-1}(j)$.
 
 We let therefore $\v_j:=\v_j^{\eps_{n(j)}}$
 and conclude by checking that the  sequence $(\v_j)_{j\in\mathbb N}$ satisfies the desired properties. Property i) is already given by \eqref{flux11}.
 Moreover, by the elementary estimates
 \[
  \|\rho_\eps\|_{W^{k,\infty}(\mathbb R^3)}\le (k+1)\,\eps^{-3-k}\,\|\rho\|_{W^{k,\infty}(\mathbb R^3)},
  \qquad k=0,1,2,
 \]
 we get
 \begin{equation}\label{limith}\begin{aligned}
 \|\v^{\eps_{n(j)}}\|_{W^{k,\infty}(\Omega')}&\le \|\v\|_{L^1(\Omega')}\,\|\rho_{\eps_{n(j)}}\|_{W^{k,\infty}(\mathbb R^3)}\le (k+1)\,\eps_{n(j)}^{-3-k} \, \|\v\|_{L^1(\Omega')}\, \|\rho\|_{W^{k,\infty}(\mathbb R^3)},
\\&= c^{-3-k}\,(k+1)\,{n(j)}^{3+k} \, \|\v\|_{L^1(\Omega')}\, \|\rho\|_{W^{k,\infty}(\mathbb R^3)},\qquad k=0,1,2,
\end{aligned} \end{equation}
therefore
 \begin{equation}\label{exte}\begin{aligned}
h_j\|\v^{\eps_{n(j)}}\|_{W^{1,+\infty}(\Omega')}
&\le{ 2h_j}\,c^{-4}\,{{n(j)}^{4}}\,\|\v\|_{L^1(\Omega')} \|\rho\|_{W^{1,\infty}(\mathbb R^3)}
\\&\le2c^{-4}\,h_j^{3/5}\,\|\v\|_{L^1(\Omega')} \|\rho\|_{W^{1,\infty}(\mathbb R^3)},
\end{aligned}\end{equation}
 so that
 $ h_j\|\v^{\eps_{n(j)}}\|_{W^{1,+\infty}(\Omega')}$ vanishes as $j\to+\infty$ and then \eqref{flux33} implies property ii).
On the other hand,  \eqref{limith} similarly implies
\begin{equation}\label{exte2}
h_j \|\v^{\eps_{n(j)}}\|_{W^{1,\infty}(\Omega')}\|\v^{\eps_{n(j)}}\|_{L^{\infty}(\Omega')}\le 2c^{-7}\,h_j^{3/10}\,\|\v\|_{L^1(\Omega')}^2 \|\rho\|_{W^{2,\infty}(\mathbb R^3)}^2
\end{equation} 
and
\begin{equation}\label{exte3}
h_j \|\v^{\eps_{n(j)}}\|_{W^{1,\infty}(\Omega')}\|\v^{\eps_{n(j)}}\|_{W^{2,\infty}(\Omega')}\le 6c^{-9}\,h_j^{1/10}\, \|\v\|_{L^1(\Omega')}^2 \|\rho\|_{W^{2,\infty}(\mathbb R^3)}^2.
 \end{equation}
 Thanks to \eqref{exte}, \eqref{exte2} and \eqref{exte3},  from \eqref{flux22} and \eqref{flux44}
 we obtain 
 \[
 \lim_{j\to+\infty} \|\v_j-\v^{\eps_{n(j)}}\|_{W^{1,\infty}(\Omega)}=0.
 \]
This entails, since $\v^\eps\to \v$ in $H^1(\Omega,\mathbb R^3)$ as $\eps\to 0$ and since $\eps_{n(j)}\to 0$ as $j\to+\infty$,
 \[
 \lim_{j\to+\infty} \|\v_j-\v\|_{H^{1}(\Omega)}\le  \lim_{j\to+\infty} \|\v_j-\v^{\eps_{n(j)}}\|_{H^{1}(\Omega)}+ \lim_{j\to+\infty} \|\v-\v^{\eps_{n(j)}}\|_{H^{1}(\Omega)}=0
 \]
 thus proving iii).
%
%
%
%
%
%
\end{proof}

\begin{lemma}\lab{upbd}{\bf (Upper bound)}. Assume \eqref{OMEGA}, \eqref{infty}, \eqref{framind}, \eqref{Z1}, \eqref{reg},    
\eqref{coerc}. Let $(h_j)_{j\in\mathbb N}\subset(0,1)$ be a vanishing sequence. For every $\v\in W^{1,p}(\Omega,\mathbb R^3)$ there exists a sequence $(\v_j)_{j\in\mathbb N}\subset W^{1,p}(\Omega,\mathbb R^3)$ such that $\v_j\wconv \v \hbox { \rm  weakly in}\  W^{1,p}(\Omega,\mathbb R^3)$ as $j\to+\infty$ and
\[ \limsup_{j\to +\infty} {\mathcal F}_{h_j}^I(\v_j) \le {\mathcal E}^I(\v).\]
\end{lemma}
\begin{proof}
It is enough to prove the result in case $\v\in H^1_{\dv}(\Omega)$. 
Let us define $\mathcal E:H^1(\om,\R^3)\to\mathbb R$ as
\[
\mathcal E(\u):=\frac12\int_\Omega \mathbb E(\u)\,D^2\mathcal W(x,\mathbf I)\,\mathbb E(\u)\,dx-\mathcal L(\u).
\]
We take the sequence $(\v_j)_{j\in\mathbb N}$ from Lemma \ref{Kato}. 
 Property ii) of Lemma \ref{Kato} yields   $\mathbf I+h_j\nabla\v_j\in\mathcal U$ for a.e. $x$ in $\Omega$ if $j$ is large enough, where $\mathcal U$ is the neighbor of $SO(3)$ that appears in \eqref{reg}.    
 In particular, $D^2\mathcal W(x,\cdot)\in C^2(\mathcal U)$ for a.e. $x\in\om$ 
 and we make use of \eqref{regW} together with
  $\det(\mathbf I+h_j\nabla\v_j)=1$ to obtain
\begin{equation*}\label{nee}\begin{aligned}
\limsup_{j\to+\infty} \mathcal |\mathcal F^I_{h_j}(\v_j)-\mathcal E(\v_j)|&\le\limsup_{j\to+\infty}\int_{\Omega} \left|\frac{1}{h_j^2}\mathcal W^I(x,\mathbf I+h_j\nabla\v_j)-\frac12\,\nabla\v_j^T D^2\mathcal W(x,\mathbf I) \nabla\v_j\right|\,dx\\
&=\limsup_{j\to+\infty}\int_{\Omega} \left|\frac{1}{h_j^2}\mathcal W(x,\mathbf I+h_j\nabla\v_j)-\frac12\,\nabla\v_j^T D^2\mathcal W(x,\mathbf I) \nabla\v_j\right|\,dx\\
&\le\limsup_{j\to+\infty}\int_{\Omega}\omega(h_j|\nabla\v_j|)\,|\nabla\v_j|^2\,dx\\&\le\limsup_{j\to+\infty}\,\|\omega(h_j\nabla \v_j)\|_{L^\infty(\om)}\int_{\om} |\nabla\v_j|^2\,dx =0.
\end{aligned}\end{equation*}
The limit in the last line is zero since $h_j\nabla \v_j\to0$ in $L^\infty(\Omega)$, since $\omega$ is increasing with $\lim_{t\to0^+}\omega(t)\to 0$  and since $(\v_j)_{j\in\mathbb N}$ is converging in $H^1(\Omega)$ as $j\to+\infty$ by Lemma \ref{Kato} .
But the  $H^1(\Omega)$ convergence also entails $\mathcal E(\v_j)\to\mathcal E(\v)$ as $j\to+\infty$.
Hence, 
\[\begin{aligned}
\limsup_{j\to+\infty}
|\mathcal F^{I}_{h_j}(\v_j)- \mathcal E^I(\v)|
&=\limsup_{j\to+\infty} |\mathcal F_{h_j}^{I}(\v_j)- \mathcal E(\v)|\\&\le \limsup_{j\to+\infty} |\mathcal F_{h_j}^{I}(\v_j)- \mathcal E(\v_j)|+\limsup_{j\to+\infty}
|\mathcal E(\v_j)- \mathcal E(\v)|=0.
\end{aligned}\]
Therefore, along the sequence $(\v_j)_{j\in\mathbb N}$ provided by Lemma \ref{Kato}, we get $\mathcal F_{h_j}^I(\v_j)\to\mathcal E^I(\v)$ as $j\to+\infty$. The result is proven.
\end{proof}
\section{ Convergence of minimisers}
We show that $\mathcal E^I$ and $\mathcal F^I$ have the same minimizers thus concluding the proof of the main result.
\begin{lemma}
\label{linel} Assume \eqref{OMEGA},\eqref{globalequi}, \eqref{comp}, \eqref{infty}, \eqref{framind}, \eqref{Z1}, \eqref{reg},    
\eqref{coerc}.
 On $W^{1,p}(\Omega,\mathbb R^3)$ we have
\beeq\lab{equalmin}
 \min {\mathcal F}^I\ =\ \min\mathcal E^I
\eneq
and
\begin{equation}
\label{equivmin}
\displaystyle\argmin {\mathcal F}^I\ =\ \argmin{\mathcal E^I}.
\end{equation}
\end{lemma}
\begin{proof} 
Existence of minimizers of $\mathcal E^I$ on $W^{1,p}(\om,\R^3)$ follows by standard arguments. Indeed, \eqref{ellipticity} and \eqref{elle} imply that a  minimizing sequence $(\u_n)_{n\in\mathbb N}\subset H^1_{\mathrm{div}}(\om,\R^3)$ of $\mathcal E^I$ satisfies $\sup_{n\in\mathbb N}\|\mathbb E(\u_n)\|_{L^2(\om)}<+\infty$, and  Lemma \ref{corcurl} entails the existence of $\mathbf u\in H^1(\om,\mathbb R^3)$ such that up to subsequences  $\mathbb E(\u_n)\to\mathbb E(\mathbf u)$ weakly in $L^2(\om,\R^{3\times 3})$. By \eqref{kornpoi} and \eqref{globalequi} we deduce $\mathcal L(\u_n)\to\mathcal L(\u)$, up to subsequences, as $n\to+\infty$. By the weak $L^2(\om)$ lower semicontinuity of $\mathbf F\mapsto\int_\om \mathbf F^T\,D^2\mathcal W(x,\mathbf I)\,\mathbf F\,dx$ we deduce that $\u$ is a minimizer of $\mathcal E^I$ over $W^{1,p}(\om,\mathbb R^3)$. 
Moreover, first order minimality conditions show that all the minimizers of $\mathcal E^I$ have the same infinitesimal strain tensor. In particular, minimizers of $\mathcal E^I$ are unique up to rigid displacements. 
\KKK

 By taking into account that
${\mathcal F}^I(\v)\leq\mathcal E^I(\v)$ for every $\v\in H^1(\om)$, and setting
$\,\mathbf \z_{\mathbf W}(x):=\textstyle\frac{1}{2}\mathbf W^{2}x$
for every $\mathbf W\in \mathbb R^{3\times 3}_{\rm skew}$, we get
$\, \E(\z_{\mathbf W})= \frac{1}{2}\mathbf W^{2}\,$ and
\begin{equation*}
\label{minFvsminE}\begin{aligned}
\displaystyle
\min_{\v\in H^1(\om)} \mathcal E^I(\v) \ &\geq \
\inf_{\v\in H^1(\om)} {\mathcal F}^I(\v)
=
\inf_{\v\in H^1(\om)}\left\{\min_{\mathbf W\in \R^{3\times 3}_{\rm skew}}\left\{\displaystyle\int_\Omega \mathcal Q^I( x, \,\mathbb E(\v)-\textstyle\frac{1}{2}\mathbf W^{2})\, dx- \mathcal L(\v) \right\}\right\}
\\
& =
\displaystyle
\min_{\mathbf W\in \mathbb R^{3\times 3}_{\rm skew}}\left\{\min_{\v\in H^1(\om)}\left\{\displaystyle\int_\Omega \mathcal Q^I(x, \,\mathbb E(\v)-\textstyle\frac{1}{2}\mathbf W^{2})\, dx- \mathcal L(\v) \right\}\right\}\\
&=
\displaystyle\min_{\mathbf W\in\mathbb R^{3\times 3}_{\rm skew}}\left\{\min_{\v\in H^1(\om)}\left\{\displaystyle\int_\Omega \mathcal Q^I(x, \,\mathbb E(\v-\z_{\mathbf W}))\, dx- \mathcal L(\v-\z_{\mathbf W})-\mathcal L(\z_{\mathbf W}) \right\}\right\}\\
&=\displaystyle\min_{\z\in H^1(\om)}\mathcal E^I(\z)\ \,-\max_{\mathbf W\in \R^{3\times 3}_{\rm skew}}
\mathcal L(\z_{\mathbf W})
\ \ge \ \min_{ H^1(\om)}\mathcal E^I
\end{aligned}
\end{equation*}
where last inequality follows by $\mathcal L(z_{\mathbf W})\leq 0$.
Therefore also $\min\mathcal F^I$ exists on $W^{1,p}(\om,\R^3)$ and  \eqref{equalmin} is proved so we are left to show \eqref{equivmin}.

First assume $\v\in \argmin \mathcal F^I$ and let
\begin{equation}\label{Wv}
\displaystyle\mathbf W_{\v}\in \argmin \left\{\int_{\om}\mathcal Q^I\Big( x,\, \mathbb E(\v)-\textstyle\frac{1}{2}\mathbf W^{2}\Big)\, dx: \ \mathbf W\in \R^{3\times 3}_{\rm skew}\right\}.
\end{equation}
If $\mathbf W_{\v}\neq \mathbf 0$ then, by setting $\z_{\mathbf W_{\v}}(x)=\frac{1}{2}\mathbf W_{\v}^{2}\,x$
we get $\mathbb E(\z_{\mathbf W_\v})=\nabla\z_{\mathbf W_\v}= \frac{1}{2}\mathbf W_\v^{2}$ and, by compatibility \eqref{comp} we obtain
\begin{equation}\label{6.6}
\begin{array}{ll}
&\displaystyle\ \min \mathcal F^I= \mathcal F^I(\v)= \int_{\om}\mathcal Q^I\Big(x,\,\mathbb E(\v-\textstyle\z_{\mathbf W_{\v}})\Big)\, dx-\mathcal L(\v- \z_{\mathbf W_{\v}})-\mathcal L(\z_{\mathbf W_{\v}})\,= \vspace{0.2cm}\\
&\displaystyle
\mathcal E^I (\v- \z_{\mathbf W_{\v}})\,-\, \mathcal L(\z_{\mathbf W_{\v}})
\geq\ \min\, \mathcal E^I\,-\, \mathcal L(\z_{\mathbf W_{\v}}) \,>\,  \min\, \mathcal E^I\,,
\end{array}
\end{equation}
 a contradiction. Therefore $\mathbf W_{\v}=\mathbf 0$, $\z_{\mathbf W_{\v}}=\mathbf 0$,
and all the inequalities in \eqref{6.6} turn out to be equalities, hence we get ${\mathcal F}^I(\v)=\mathcal E^I(\v)=\min \mathcal E^I=\min {\mathcal  F}^I$,
therefore $\v\in \argmin\mathcal E$
and $\argmin{\mathcal F}^I \subseteq \argmin\mathcal E^I.$
 In order to show the opposite inclusion, we assume $\v\in \argmin \mathcal E^I$ and still referring to the choice \eqref{Wv} we get $2\dv\v=0=\tr \mathbf W_{\v}^2=-|\mathbf W_{\v}|^2$. 
Therefore $\mathcal E^I(\v)={\mathcal F}^I(\v)$ and $\v\in \argmin{\mathcal F}^I$.
\end{proof}
\begin{remark}\rm
{The proof of lemma \ref{linel} shows that, although Theorem \ref{mainth1} is not true if \eqref{comp} is replaced by the weaker condition
$$\mathcal L(\mathbf W^2 x)\le 0 \qquad
\forall\  \mathbf W\in \R^{3\times 3}_{\rm skew},$$ still  $\mathcal E^I$ and $\mathcal F^I$ have the same minimal values under such weaker condition. As an example,  we may consider} $\mathbf f$  and $\mathbf g$ as in Remark \ref{phi}, but with $\int_\om\phi(x)\,dx=\lambda|\Omega|$ instead of $\int_\om\phi(x)\,dx<\lambda|\Omega|$. In this case $\mathcal L(\mathbf W^2x)=0$ for any $\mathbf W\in\mathbb R^{3\times3}_{\rm skew}$ and then Theorem \ref{mainth1} does not apply (see Remark \ref{rmk2.2}). However, still $\mathcal E^I(\v)$ is minimal (with minimal value $0$) if and only if $\v$ is a rigid displacement. Moreover, $\mathcal F^I$ is minimal on rigid displacements as well.
\end{remark}
\KKK

\bigskip

\begin{proofth1}
We obtain \eqref{convmin}  from Lemma \ref{lemmabound}. If $(\mathbf v_j)_{j\in\mathbb N}\subset W^{1,p}(\om,\mathbb R^3)$ is a sequence  such that
\[ \label{assinf}  \lim_{j\to +\infty}\left( \mathcal F^{I}_{h_{j}}(\v_{j})-\inf_{W^{1,p}(\om,\mathbb R^3)}\mathcal F^{I}_{h_{j}}\right)= 0,
\]
then by Lemma \ref{compactness} there exists $\v_*\in H^1(\om)$ such that, up to subsequences, $\mathbb E(\v_j)\wconv \mathbb E(\v_*)$ weakly in $L^p(\om)$ hence by Lemma \ref{lowerbd} 
\begin{equation}\lab{liminf}
\displaystyle\liminf_{j\to +\infty}{\mathcal F}^I_{h_j}(\v_j)\ge {\mathcal F}^I(\v_*).
\end{equation}
On the other hand by Lemma \ref{upbd}  for every $\v\in H^1(\om)$ there exists a sequence $(\v_j)_{j\in\mathbb N}\subset W^{1,p}(\om)$ such that 
$\mathbb E(\v_j)\wconv \mathbb E(\v_*)$ weakly in $L^p(\om)$ and 
\begin{equation}\lab{limsup}
\displaystyle\limsup_{j\to +\infty}{\mathcal F}^I_{h_j}(\v_j)\le {\mathcal E}^I(\v),
\end{equation}
that is, ${\mathcal F}^I(\v_*)\le {\mathcal E}^I(\v)$ for every $\v\in H^1(\om)$. Hence,   ${\mathcal F}^I(\v_*)\le \min_{W^{1,p}(\om)} {\mathcal E}^I$ and by Lemma  \ref{linel} we obtain ${\mathcal F}^I(\v_*)\le \min_{W^{1,p}(\om)} {\mathcal E}^I=\min_{W^{1,p}(\om)} {\mathcal F}^I$ so that by \eqref{liminf}, \eqref{limsup} 
\begin{equation*}
\displaystyle\lim_{j\to +\infty}{\mathcal F}^I_{h_j}(\v_j)={\mathcal F}^I(\v_*)=\min_{W^{1,p}(\om)} {\mathcal F}^I= \min_{W^{1,p}(\om)} {\mathcal E}^I
\end{equation*}
thus completing the proof.
\end{proofth1}

\subsection*{Acknowledgements} 
The authors acknowledge support from the MIUR-PRIN  project  No 2017TEXA3H.
The authors are members of the
GNAMPA group of the Istituto Nazionale di Alta Matematica (INdAM).


\begin{thebibliography}{99}



\bibitem{ABK}{V. Agostiniani, T. Blass, K. Koumatos, }{\textit{From nonlinear to linearized elasticity via Gamma-convergence: the case of multiwell energies satisfying weak coercivity conditions, }}
Math. Models Methods Appl. Sci. \textbf{25}  (2015), 1--38.

\bibitem{ADMDS}{V. Agostiniani, G. Dal Maso, A. De Simone, }{\textit{Linear elasticity obtained from finite elasticity by Gamma-convergence under weak coerciveness conditions, }}Ann. Inst. H. Poincar\'e Anal. non Lin\'eaire, \textbf{29}  (2012), 715--735.




\bibitem{ADMLP}{R. Alicandro, G. Dal Maso, G. Lazzaroni, M. Palombaro,}{\textit{ Derivation of a linearised elasticity model from singularly perturbed multiwell energy functionals}}, Arch. Ration. Mech. Anal. {\bf 230} (2018), 1--45.


\bibitem{ABP}
G. Anzellotti, S. Baldo, D. Percivale, \textit{Dimension reduction in variational problems, asymptotic development
in $\Gamma$-convergence and thin structures in elasticity}, Asympt. Analysis {\bf 9} (1994), 61-100. 

\bibitem{AB} E. M. Arruda, M. C. Boyce, \textit{A three-dimensional constitutive model for the large stretch behavior of rubber elastic materials}. J. Mech.  Phys.  Solids {\bf 41}, no. 2 (1993), 389--412.




\bibitem{BA}  M. C. Boyce,  E. M. Arruda \textit{Constitutive Models of Rubber Elasticity: A Review}, Rubber Chemistry and Technology {\bf 73}, no. 3 (2000), 504--523. 

















\bibitem{Ch}
E. W. Chaves, {\textit{Notes on Continuum Mechanics}}, Springer, Berlin, 2013.



\bibitem{C} P. G. Ciarlet,  \textit{Mathematical Elasticity, Volume I: Three-Dimensional Elasticity}, Elsevier, 1988.

\bibitem{CF}
C. J. Chuong, Y. C. Fung,  \textit{Compressibility and constitutive equation of arterial wall in radial compression experiments},  J. Biomech. {\bf 17}, no. 1 (1984), 35--40.



\bibitem{CD} S. Conti, G. Dolzmann, \textit{$\Gamma$-convergence for incompressible elastic plates}, 
Calc. Var. {\bf 34}  (2009), 531--551.


\bibitem{DMPN} G. Dal Maso, M. Negri, D. Percivale,  \textit{Linearized elasticity as $\Gamma$-limit of finite elasticity}, Set-Valued Anal. \textbf{10}, no. 2-3 (2002), 165-183.












\bibitem{FJM0} G. Friesecke, R. D. James, S. M\"uller, \textit{ A theorem on geometric rigidity and the derivation of non linear plate theory from three dimensional elasticity}, Comm.Pure Appl. Math. {\bf 55} (2002), 1461--1506.


\bibitem{FJM} G. Friesecke, R. D. James, S. M\"uller, \textit{A Hierarchy of Plate Models from Nonlinear Elasticity by Gamma-Convergence},
Arch. Rational Mech. Anal. 1{\bf 80} (2006), 183-236.






\bibitem{G} M. E. Gurtin, {\it An Introduction to Continuum Mechanics}, Academic Press, New York, 1981.


%
%

\bibitem{HP} A. Henrot, M. Pierre, {\it Shape variation and optimization: a geometrical analysis}, EMS Tracts in Mathematics, 2018.

\bibitem{HA} G.A. Holzapfel, \textit{Nonlinear Solid Mechanics: A Continuum Approach for Engineering}, Wiley, Chichester, 2000.


\bibitem{HO}
  G. A. Holzapfel, R. E. Ogden (eds), \textit{Biomechanics of soft tissue in cardiovascular system}, Springer, New York, 2003.

%


\bibitem{JS}
M. Jesenko, B. Schmidt, {\it 
Geometric linearization of theories for incompressible elastic materials and applications}, preprint arXiv:2004.11271


\bibitem{KS}  P.  Kalita, R. Schaefer, \textit{Mechanical Models of Artery Walls}, Arch. Comput. Methods Eng. {\bf 15} (2008), 1--36.

\bibitem{Kato}
T. Kato, M. Mitrea, G. Ponce,  M. Taylor, {\it
Extension and representation of divergence-free vector fields on bounded domains},
Mathematical Research Letters {\bf 7} (2000), 643--650.




\bibitem{LM} M. Lecumberry, S. M\"uller, \textit{Stability of slender bodies under compression and validity of
von K\'{a}rm\'{a}n theory}, Arch. Rational Mech. Anal. {\bf 193} (2009), 255-310.


\bibitem{Lo} A.E. Love, {\it A Treatise on the Mathematical Theory of Elasticity}, Dover, 1944.







%

\bibitem{MPTJOTA} F. Maddalena, D. Percivale, F. Tomarelli,
\textit{A new variational approach to linearization of traction problems in elasticity
,}\textit{ J.Optim.Theory Appl.}  {\bf 182} (2019), 383--403.

%
%
   
  \bibitem{MPTARMA} F. Maddalena, D. Percivale, F. Tomarelli, \textit{The gap in pure traction problems between linear elasticity and variational limit of finite elasticity}, Arch. Ration. Mech. Anal.  {\bf 234} (2019), 1091--1120.

    

%
%
%

\bibitem{MP}
E. Mainini, D. Percivale,
{\it Linearization of elasticity models for incompressible materials}, preprint arXiv:2004.09286



\bibitem{MV}  G. Marckmann, E. Verron, \textit{Comparison of hyperelastic models for rubberlike materials}, Rubber Chemistry and Technology {\bf 79}, no. 5 (2006), 835--858.




\bibitem{N} J. A. Nitsche, \textit{On Korn’s second inequality}, RAIRO Anal. Num\'er. {\bf 15} (1981) 237--248.

\bibitem{og2}  R. W. Ogden, \textit {Mechanics of Rubberlike Solids}. In: Gutkowski W., Kowalewski T.A. (eds), Mechanics of the 21st Century. Springer, Dordrecht, 2005.

\bibitem{OGDEN} R. W. Ogden, \textit{Non-linear elastic deformations}, Dover, New York, 1997.

\bibitem{og1}
R. W. Ogden, \textit {Recent advances in the phenomenological theory of rubber elasticity}, Rubber Chem. Technol. {\bf 59} (1986), 361--383.



\bibitem{PT1}{ D. Percivale, F. Tomarelli, }{\textit{Scaled
Korn-Poincar\'e inequality in BD and a model of elastic plastic
cantilever}, } {Asymptot. Anal. }{{\textbf{23}}, no. 3-4 }{(2000), 291--311. }

\bibitem{PTplate}{D. Percivale, F. Tomarelli, }{\textit {From SBD to SBH: the elastic-plastic plate, }}
Interfaces Free Boundaries {{\textbf 4}, no. 2} (2002), 137-165.

\bibitem{PT2} {D. Percivale, F. Tomarelli, }{\textit{A variational principle for plastic hinges in a beam},}{  Math. Models Methods Appl. Sci. {\bf 19}, no. 12, (2009), 2263-2297.}

\bibitem{PT4} {D. Percivale, F. Tomarelli, }
            {\textit{Smooth and broken minimizers of some free discontinuity problems},}
       in: P. Colli et al. (eds.), \textit{Solvability, Regularity, and Optimal Control of Boundary Value Problems for PDEs}, Springer INdAM Series \textbf{22} (2017), 431--468.

\bibitem{PPG} {P. Podio-Guidugli, }{\textit {On the validation of theories of thin elastic structures}}, Meccanica \textbf{49}, no.  6 (2014), 1343-1352.



\bibitem{SO}
G. Saccomandi, R. W.  Ogden (eds), \textit{Mechanics and thermomechanics of rubberlike solids}. CISM Courses and Lectures, Vol. 452.  Springer-Verlag Wien, 2004.

\bibitem{Sc}
B. Schmidt, {\it Linear Gamma-limits of multiwell energies in nonlinear elasticity theory}, Contin. Mech. Thermodyn. {\bf 20}, no. 6 (2008), 375–396.







\bibitem{S}
W. S. Slaughter, \textit{The Linearized Theory of Elasticity}, Birkh\"auser Basel, 2002.

\bibitem{SHP}
P. Steinmann, M. Hossain, G.  Possart, \textit{Hyperelastic models for rubber-like materials: consistent tangent operators and suitability for Treloar’s data}, Arch. Appl. Mech. {\bf 82} (2012) 1183--1217. 


%
%















\bibitem{Y} O. H. Yeoh, \textit{Characterization of Elastic Properties of Carbon-Black-Filled Rubber Vulcanizates}, Rubber Chemistry and Technology {\bf 63}, no. 5 (1990),  792--805. 

\bibitem{Y2}
O. H. Yeoh, \textit{Some forms of the strain energy function for rubber}, Rubber Chem.  Tech. {\bf 66}, no. 5 (1993),  754--771.




\KKK












 









 \end{thebibliography}
\end{document}